\documentclass[12pt]{amsart}

\usepackage[utf8]{inputenc}
\usepackage[T1,T2A]{fontenc}
\usepackage[english]{babel}
\usepackage{amsfonts}
\usepackage{amsbsy}
\usepackage{amssymb}
\usepackage{graphicx}
\usepackage{mathrsfs}
\usepackage[colorlinks,breaklinks,allcolors=blue]{hyperref}
\usepackage{longtable}
\usepackage{enumitem}
\usepackage{tikz}
\usepackage{comment}

\catcode`~=11 % make LaTeX treat tilde (~) like a normal character
\newcommand{\urltilde}{\kern -.15em\lower .7ex\hbox{~}\kern .04em}
\catcode`~=13 % revert back to treating tilde (~) as an active character

\renewcommand{\abovecaptionskip}{0pt}
\renewcommand{\belowcaptionskip}{6pt}
\makeatletter

\renewcommand{\@makecaption}[2]{
\vspace{\abovecaptionskip}%
\sbox{\@tempboxa}{#1. #2}%
\global\@minipagefalse \hbox to \hsize {{\scshape \hfil #1.
#2\hfil}} \vspace{\belowcaptionskip}}

\makeatother

\newcommand{\Hom}{\operatorname{Hom}}
\newcommand{\rk}{\operatorname{rk}}

\newcommand{\Ker}{\operatorname{Ker}}

\newcommand{\Ad}{\operatorname{Ad}}

\newcommand{\GL}{\operatorname{GL}}

\newcommand{\SL}{\operatorname{SL}}
\newcommand{\Sp}{\operatorname{Sp}}
\newcommand{\Spin}{\operatorname{Spin}}
\newcommand{\SO}{\operatorname{SO}}

\newcommand{\Supp}{\operatorname{Supp}}
\newcommand{\Lie}{\operatorname{Lie}}

\newcommand{\ZZ}{\mathbb Z}
\newcommand{\QQ}{\mathbb Q}

\newcommand{\GG}{\mathbb G}
\newcommand{\KK}{\mathbb K}

\newtheorem{theorem}{Theorem}

\newtheorem{proposition}[theorem]{Proposition}
\newtheorem{lemma}[theorem]{Lemma}

\newtheorem*{question*}{Question}

\theoremstyle{definition}

\newtheorem{example}[theorem]{Example}

\theoremstyle{remark}

\newtheorem{remark}[theorem]{Remark}

\makeatletter

\@addtoreset{theorem}{section}

\makeatother

\numberwithin{equation}{section}

\makeatletter

\@addtoreset{theorem}{section}

\makeatother

\numberwithin{equation}{section}

\newcounter{num}[table]
\newcommand{\no}{\refstepcounter{num}\arabic{num}}

\newcounter{alg}
\newcounter{stepalg}[alg]

\newcommand{\newalg}{\refstepcounter{alg}\Alph{alg}}
\newcommand{\step}{\refstepcounter{stepalg}\Alph{alg}\arabic{stepalg}}

\renewcommand{\arraystretch}{1.2}

\begin{document}

\renewcommand{\proofname}{Proof}
\renewcommand{\abstractname}{Abstract}
\renewcommand{\refname}{References}
\renewcommand{\figurename}{Figure}
\renewcommand{\tablename}{Table}

\title[On computing the spherical roots for a class of spherical subgroups]
{On computing the spherical roots\\ for a class of spherical subgroups}

\author{Roman Avdeev}

%\thanks{}

\address{%
{\bf Roman Avdeev} \newline\indent HSE University, Moscow, Russia}

\email{suselr@yandex.ru}

%\date{\today}



\thanks{This paper is an output of a research project implemented as part of the Basic Research Program at HSE University.}

\begin{abstract}
Given a connected reductive algebraic group~$G$, we consider the class of spherical subgroups $H \subset G$ such that $H$ is regularly embedded in a parabolic subgroup $P \subset G$ and $H,P$ have a common Levi subgroup~$L$.
In a previous paper, the author developed a fast algorithm that reduces the computation of the set of spherical roots for such subgroups~$H$ to the case where the quotient of Lie algebras $\Lie P / \Lie H$ is a strictly indecomposable spherical $L$-module.
In this paper, we complete the classification of all such cases and compute the spherical roots for each of them, which enables one to use the above fast algorithm directly for computing the spherical roots for arbitrary spherical subgroups in the class under consideration.
\end{abstract}

\maketitle

\section{Introduction}

Throughout this paper, we work over an algebraically closed field~$\KK$ of characteristic zero.

Let $G$ be a connected reductive algebraic group and let $X$ be a $G$-variety, that is, an algebraic variety equipped with a regular action of~$G$.
The $G$-variety $X$ is said to be \textit{spherical} if it is irreducible, normal, and possesses an open orbit with respect to the induced action of a Borel subgroup $B \subset G$.
An algebraic subgroup $H \subset G$ is said to be \textit{spherical} if $G/H$ is a spherical homogeneous space.

Every spherical variety $X$ contains an open $G$-orbit $O$ and hence can be regarded as a $G$-equivariant open embedding of~$O$.
Given a spherical homogeneous space $G/H$, the Luna--Vust theory (see~\cite{LV} and also~\cite{Kn91}) provides a combinatorial description of all (normal) $G$-equivariant open embeddings of $G/H$ based on the following three combinatorial invariants of~$G/H$: the \textit{weight lattice} $\Lambda_G(G/H)$, the finite set $\Sigma_G(G/H)$ of \textit{spherical roots}, and the finite set $\mathcal D_G(G/H)$ of \textit{colors}.
Moreover, all such embeddings are parameterized by combinatorial objects called \textit{colored fans}, which generalize usual fans used for classifying toric varieties.

It turns out that the three above-mentioned invariants also play a crucial role in the classification of spherical homogeneous spaces themselves.
Namely, extending the datum of $\Lambda_G(G/H)$, $\Sigma_G(G/H)$, $\mathcal D_G(G/H)$ with the parabolic subgroup $P_G(G/H) \subset G$ defined as the stabilizer of the open $B$-orbit in~$G/H$ one obtains a quadruple that uniquely determines $G/H$ up to a $G$-equivariant isomorphism.
Moreover, determining relations between the four invariants one obtains a complete classification of spherical homogeneous spaces in terms of combinatorial objects called \textit{homogeneous spherical data}.
This approach was initiated in~\cite{Lu01} and subsequently completed by a joint effort of several researchers; see~\cite{Lo1,BP14,BP15,BP16}.

In view of the importance of the invariants $\Lambda_G(G/H)$, $\Sigma_G(G/H)$, $\mathcal D_G(G/H)$ in the classification of spherical varieties, a natural problem is to compute them starting from an explicit form of a spherical subgroup~$H \subset G$.
A standard way of specifying a subgroup $H$ is to use a regular embedding $H \subset P$ in a parabolic subgroup~$P \subset G$, where ``regular'' means the inclusion $H_u \subset P_u$ of the unipotent radicals of $H$ and~$P$, respectively.
If $L$ is a Levi subgroup of~$P$, then up to conjugacy we may also assume that $K = H \cap L$ is a Levi subgroup of~$H$.
In this setting, it was shown in~\cite[\S\,2.3]{Avd_SSSS} that the computation of all the four above-mentioned invariants reduces to computing~$\Sigma_G(G/H)$ and one more invariant~$\widehat \Lambda^+_G(G/H)$, called the \textit{extended weight monoid}, which is closely related to the colors.
Further, the results of~\cite{Avd_EWM} show that the computation of $\widehat \Lambda^+_G(G/H)$ can be reduced to that of~$\Sigma_G(G/H)$.
Consequently, computing all the invariants of spherical homogeneous spaces reduces to computing the spherical roots.

As for computing the set $\Sigma_G(G/H)$, explicit solutions for this problem are known for the cases where $H$ is reductive (see~\cite{BP15}) or strongly solvable, that is, $H$ is contained in a Borel subgroup of~$G$ (see~\cite{Avd_SSSS}).
Apart from that, a general strategy for computing $\Sigma_G(G/H)$ was proposed in~\cite{Avd_DSS}.
This strategy is based on finding one-parameter degenerations of the Lie algebra $\Lie H$ in the Lie algebra $\Lie G$ having special properties.
The strategy was also implemented for subgroups $H$ regularly embedded in a parabolic subgroup $P \subset G$ such that Levi subgroups $K \subset H$ and $L \subset P$ satisfy $L' \subset K \subset L$, where $L'$ is the derived subgroup of~$L$.
The result of this implementation is a recursive algorithm that reduces the computation of $\Sigma_G(G/H)$ to the same problem for two other spherical subgroups $N_1,N_2 \subset G$ after performing two degenerations of $\Lie H$ in~$\Lie G$.
These subgroups satisfy $|\Sigma_G(G/N_1)| = |\Sigma_G(G/N_2)| = |\Sigma_G(G/H)| - 1$ and $\Sigma_G(G/H) = \Sigma_G(G/N_1) \cup \Sigma_G(G/N_2)$.
We call this algorithm the \textit{base algorithm}; in general it has exponential complexity depending on the value $r = |\Sigma_G(G/H)|$.
In addition, for groups $H$ satisfying $K = L$ a significantly faster algorithm was developed in the same paper~\cite{Avd_DSS}; we call it the \textit{optimized algorithm}.
This algorithm, which conjecturally has linear complexity, reduces the computation of $\Sigma_G(G/H)$ to the same problem for a finite number of new spherical subgroups~$H_1,\ldots,H_m$ such that for each $i = 1,\ldots, m$ the subgroup $H_i$ is standardly embedded in a parabolic subgroup $P_i \subset G$, the groups $H_i,P_i$ have a common Levi subgroup~$L_i$, and the quotient of Lie algebras $\Lie P_i / \Lie H_i$ is a strictly indecomposable spherical $L_i$-module (see the definition in~\S\,\ref{subsec_SM_gen}).
It is known that such modules can have no more than two simple summands, and all the cases with exactly one simple summand were also classified in~\cite{Avd_DSS}.
Moreover, it turned out that the sets of spherical roots for all such cases were already known.
We reproduce this classification in Theorem~\ref{thm_par1psi1}.

In this paper, we classify all spherical subgroups $H$ satisfying $K = L$ and such that the $L$-module $\Lie P / \Lie H$ is strictly indecomposable and has exactly two simple components.
Moreover, for all such~$H$ we compute the corresponding sets of spherical roots.
These results are stated in Theorems~\ref{thm_par1psi2} and~\ref{thm_par2psi2}.
We remark that in the cases where $G$ is an exceptional simple group of type~$\mathsf F_4$, $\mathsf E_6$, $\mathsf E_7$, or~$\mathsf E_8$, the results are obtained using computer calculations.
In view of the discussion in the previous paragraph, our results enable one to compute the set $\Sigma_G(G/H)$ for an arbitrary spherical subgroup~$H$ satisfying $K = L$ (without the restriction on $\Lie P / \Lie H$) using only the optimized algorithm without additional calculations.

This paper is organized as follows.
In \S\,\ref{sect_prelim} we fix notation and conventions, discuss some general results, and recall all facts on spherical varieties needed in this paper.
In \S\,\ref{sect_active_C-roots} we discuss in more detail the class of spherical subgroups considered in this paper, recall the construction of degenerations for them, and reproduce the base algorithm along with the optimized algorithm.
In \S\,\ref{sect_comp_SR} we state and prove the main results of this paper.

%\subsection*{Acknowledgements}

\section{Preliminaries}
\label{sect_prelim}

\subsection{Notation and conventions}

Throughout this paper, we work over an algebraically closed field $\KK$ of characteristic zero.
All topological terms relate to the Zariski topology.
All subgroups of algebraic groups are assumed to be algebraic.
The Lie algebras of algebraic groups denoted by capital Latin letters are denoted by the corresponding small Gothic letters.
A~\textit{$K$-variety} is an algebraic variety equipped with a regular action of an algebraic group~$K$.

$\ZZ_{\ge0} = \lbrace z \in \ZZ \mid z \ge 0 \rbrace$;

$\KK^\times$ is the multiplicative group of the field~$\KK$;

$\GG_a$ is the additive group of the field~$\KK$;

$|X|$ is the cardinality of a finite set~$X$;

$\langle v_1,\ldots, v_k \rangle$ is the linear span of vectors $v_1,\ldots, v_k$ of a vector space~$V$;

$V^*$ is the space of linear functions on a vector space~$V$;

$\mathrm S^k(V)$ is the $k$th symmetric power of a vector space~$V$;

$\wedge^k(V)$ is the $k$th exterior power of a vector space~$V$;

$L^0$ is the connected component of the identity of an algebraic group~$L$;

$L'$ is the derived subgroup of a group~$L$;

$L_u$ is the unipotent radical of an algebraic group~$L$;

$Z(L)$ is the center of a group~$L$;

$\mathfrak X(L)$ is the character group (in additive notation) of an algebraic group~$L$;

$N_L(K)$ is the normalizer of a subgroup $K$ in a group~$L$;

$\KK[X]$ is the algebra of regular functions on an algebraic variety~$X$;

$\KK(X)$ is the field of rational functions on an irreducible algebraic variety~$X$;

$G$ is a connected reductive algebraic group;

$B \subset G$ is a fixed Borel subgroup;

$T \subset B$ is a fixed maximal torus;

$B^-$ is the Borel subgroup of~$G$ opposite to~$B$ with respect to~$T$, so that $B \cap B^- = T$;

$(\cdot\,,\,\cdot)$ is a fixed inner product on~$\QQ\mathfrak X(T)$ invariant with respect to the Weyl group $N_G(T)/T$;

$\Delta \subset \mathfrak X(T)$ is the root system of~$G$ with respect to~$T$;

$\Delta^+ \subset \Delta$ is the set of positive roots with respect to~$B$;

$\Pi \subset \Delta^+$ is the set of simple roots with respect to~$B$;

$\alpha^\vee \in \Hom_\ZZ(\mathfrak X(T), \ZZ)$ is the coroot corresponding to a root $\alpha \in \Delta$;

$h_\alpha \in \mathfrak t$ is the image of $\alpha^\vee$ in $\mathfrak t$ under the chain $\Hom_\ZZ(\mathfrak X(T), \ZZ) \hookrightarrow (\mathfrak t^*)^* \xrightarrow{\sim} \mathfrak t$;

$\mathfrak g_\alpha \subset \mathfrak g$ is the root subspace corresponding to a root~$\alpha \in \Delta$;

$e_\alpha \in \mathfrak g_\alpha$ is a fixed nonzero element.

The simple roots and fundamental weights of simple algebraic groups are numbered as in~\cite{Bo}.

For every $\beta = \sum \limits_{\alpha \in \Pi} k_\alpha \alpha \in \ZZ_{\ge0}\Pi$, we define its \textit{support} $\Supp \beta = \lbrace \alpha \in \Pi \mid k_\alpha > 0 \rbrace$ and \textit{height} $\operatorname{ht} \beta = \sum \limits_{\alpha \in \Pi} k_\alpha$.

We fix a nondegenerate $G$-invariant inner product on~$\mathfrak g$ and for every subspace $\mathfrak u \subset \mathfrak g$ let $\mathfrak u^\perp$ be the orthogonal complement of $\mathfrak u$ in~$\mathfrak g$ with respect to this product.

The groups $\mathfrak X(B)$ and $\mathfrak X(T)$ are identified via restricting characters from~$B$ to~$T$.

Given a parabolic subgroup $P \subset G$ such that $P \supset B$ or $P \supset B^-$, the unique Levi subgroup $L$ of $P$ containing~$T$ is called the \textit{standard Levi subgroup} of~$P$.
By abuse of language, in this situation we also say that $L$ is a standard Levi subgroup of~$G$.
The unique parabolic subgroup $Q$ of $G$ such that $\mathfrak p + \mathfrak q = \mathfrak g$ and $P \cap Q = L$ is said to be \textit{opposite} to~$P$.

Let $L \subset G$ be a standard Levi subgroup.
We put $B_L = B \cap L$, so that $B_L$ is a Borel subgroup of~$L$.
If $V$ is a simple $L$-module, by a highest (resp. lowest) weight vector of~$V$ we mean a $B_L$-semiinvariant (resp. $(B^- \cap L)$-semiinvariant) vector in~$V$.
These conventions on~$V$ are also valid if $L = G$.

Given a standard Levi subgroup $L \subset G$, we let $\Delta_L \subset \Delta$ be the root system of~$L$ and put $\Delta^+_L = \Delta^+ \cap \Delta_L$ and $\Pi_L = \Pi \cap \Delta_L$, so that $\Delta^+_L$ (resp. $\Pi_L$) is the set of all positive (resp. simple) roots of~$L$ with respect to the Borel subgroup~$B_L$.

Let $K$ be a group and let $K_1,K_2$ be subgroups of~$K$.
We write $K = K_1 \rightthreetimes K_2$ if $K$ is a semidirect product of~$K_1,K_2$ with $K_2$ being a normal subgroup of~$K$.

\subsection{Levi roots and their properties}
\label{subsec_Levi_roots}

Let $L$ be a standard Levi subgroup of~$G$ and let $C$ be the connected center of~$L$.
We consider the natural restriction map $\varepsilon \colon \mathfrak X(T) \to \mathfrak X(C)$ and extend it to the corresponding map $\varepsilon_\QQ \colon \QQ\mathfrak X(T) \to \QQ\mathfrak X(C)$.
Let $(\Ker \varepsilon_\QQ)^\perp \subset \QQ\mathfrak X(T)$ be the orthogonal complement of $\Ker \varepsilon_\QQ$ with respect to the inner product $(\cdot\,, \cdot)$.
Under the map $\varepsilon_\QQ$, the subspace $(\Ker \varepsilon_\QQ)^\perp$ maps isomorphically to $\QQ\mathfrak X(C)$; we equip $\QQ\mathfrak X(C)$ with the inner product transferred from $(\Ker \varepsilon_\QQ)^\perp$ via this isomorphism.

Consider the adjoint action of~$C$ on~$\mathfrak g$.
For every $\lambda \in \mathfrak X(C)$, let $\mathfrak g(\lambda) \subset \mathfrak g$ be the corresponding weight subspace of weight~$\lambda$.
It is well known that $\mathfrak g(0) = \mathfrak l$.
We put
\[
\Phi = \lbrace \lambda \in \mathfrak X(C) \setminus \lbrace 0 \rbrace \mid \mathfrak g(\lambda) \ne \lbrace 0 \rbrace \rbrace.
\]
Then there is the following decomposition of $\mathfrak g$ into a direct sum of $C$-weight subspaces:
\begin{equation} \label{eqn_decomposition}
\mathfrak g = \mathfrak l \oplus \bigoplus\limits_{\lambda \in \Phi} \mathfrak g(\lambda).
\end{equation}
In what follows, elements of the set $\Phi$ will be referred to as \textit{$C$-roots}.
It is easy to see that $\Phi = \varepsilon(\Delta \setminus \Delta_L)$.
In particular, $\Phi = - \Phi$.

Now consider the adjoint action of~$L$ on~$\mathfrak g$.
Then each $C$-weight subspace of $\mathfrak g$ becomes an $L$-module in a natural way.
The following proposition is well known.

\begin{proposition} \label{prop_properties_of_g(lambda)}
The following assertions hold.
\begin{enumerate}[label=\textup{(\alph*)},ref=\textup{\alph*}]
\item \label{prop_properties_of_g(lambda)_a}
\textup(see \cite[Theorem~1.9]{Ko} or \cite[Ch.~3, Lemma~3.9]{GOV}\textup)
For every $\lambda \in \Phi$, the subspace $\mathfrak g(\lambda)$ is a simple $L$-module.

\item \label{prop_properties_of_g(lambda)_b}
\textup(see \cite[Lemma~2.1]{Ko}\textup)
For every $\lambda, \mu, \nu \in \Phi$ with $\lambda = \mu + \nu$ one has $\mathfrak g(\lambda) = [\mathfrak g(\mu), \mathfrak g(\nu)]$.

\item \label{prop_properties_of_g(lambda)_c}
\textup(see \cite[Proposition~2.1(c)]{Avd_DSS}\textup)
For every $\lambda \in \Phi$ there is an $L$-module isomorphism $\mathfrak g(-\lambda) \simeq \mathfrak g(\lambda)^*$.
\end{enumerate}
\end{proposition}

\subsection{Additive degenerations of subspaces in simple \texorpdfstring{$\mathfrak{sl}_2$}{sl\_2}-modules}
\label{subsec_degen_subsp}

Consider the Lie algebra $\mathfrak{sl}_2$ with standard basis $\lbrace e,h,f\rbrace$, so that $[e,f] = h$, $[h,e] = 2e$, and $[h,f] = -2f$.
Let $V$ be a simple $\mathfrak{sl}_2$-module with highest weight $p \in \ZZ_{\ge0}$.
Fix a basis $\lbrace v_{p-2i} \mid i=0,\ldots,p \rbrace$ of $V$ such that $h \cdot v_{p-2i} = (p-2i)v_{p-2i}$ for all $i=0,\ldots, p$, $f \cdot v_{p-2i} = v_{p-2i-2}$ for all $i = 0,\ldots, p-1$, and $f \cdot v_{-p} = 0$.

Consider the one-parameter unipotent subgroup $\phi \colon \GG_a \to \GL(V)$ given by $\phi(t) = \exp (tf)$, let $U \subset V$ be a subspace with $\dim U = k$, and regard $U$ as a point in the Grassmannian of $k$-dimensional subspaces of~$V$.

\begin{proposition}[{\cite[Proposition~2.5]{Avd_DSS}}] \label{prop_limitII_prelim}
There exists the limit $\lim \limits_{t \to \infty}\phi(t) U = U_\infty$.
Moreover, $U_\infty = \langle v_{-p+2i} \mid i=0,\ldots, k-1 \rangle$.
\end{proposition}

In \S\,\ref{subsec_degen_add}, we will apply Proposition~\ref{prop_limitII_prelim} in situations where the subspace $U$ is $h$-stable.
In this case, $U = \langle v_{n_0}\rangle \oplus \ldots \oplus \langle v_{n_{k-1}} \rangle$ for some $n_0 < \ldots < n_{k-1}$ and
\[
U_\infty = \langle v_{-p} \rangle \oplus \langle v_{-p+2} \rangle \oplus \ldots \oplus \langle v_{-p+2k-2} \rangle.
\]
For describing $U_\infty$ in our applications, it will be convenient for us to use the following terminology: for every $i = 0,\ldots,k-1$ we say that $\langle v_{n_i} \rangle$ \textit{shifts to} $\langle v_{-p+2i} \rangle$ under the degeneration.
Observe that $-p+2i \le n_i$ for all~$i$.

\subsection{Spherical varieties and some combinatorial invariants of them}
\label{subsec_comb_inv_SV}

Recall from the introduction that a $G$-variety $X$ is said to be spherical if it is irreducible, normal, and has an open orbit for the induced action of~$B$.
Recall also that a subgroup $H \subset G$ is said to be spherical if $G/H$ is a spherical homogeneous space.

Let $X$ be a spherical $G$-variety.
In this subsection, we introduce several combinatorial invariants of $X$ that will be needed in our paper.

For every $\lambda \in \mathfrak X(T)$ let $\KK(X)_\lambda^{(B)}$ be the space of $B$-semiinvariant rational functions on~$X$ of weight $\lambda$.
Then the \textit{weight lattice} of $X$ is by definition
\[
\Lambda_G(X) = \lbrace \lambda \in \mathfrak X(T) \mid \KK(X)_\lambda^{(B)} \ne \lbrace 0 \rbrace \rbrace.
\]
The \textit{rank} of~$X$ is defined as $\rk_G(X) = \rk \Lambda_G(X)$.
Since $B$ has an open orbit in~$X$, it follows that for every $\lambda \in \Lambda_G(X)$ the space $\KK(X)^{(B)}_\lambda$ has dimension~$1$ and hence is spanned by a nonzero function~$f_\lambda$.

Put $\mathcal Q_G(X) = \Hom_\ZZ(\Lambda_G(X), \QQ)$.

Every discrete $\QQ$-valued valuation $v$ of the field $\KK(X)$ vanishing on $\KK^\times$ determines an element $\rho_v \in \mathcal Q_G(X)$ such that $\rho_v(\lambda) = v(f_\lambda)$ for all $\lambda \in \Lambda_G(X)$.
It is known that the restriction of the map $v \mapsto \rho_v$ to the set of $G$-invariant discrete $\QQ$-valued valuations of $\KK(X)$ vanishing on~$\KK^\times$ is injective (see~\cite[7.4]{LV} or~\cite[Corollary~1.8]{Kn91}) and its image is a finitely generated cone containing the image in~$\mathcal Q_G(X)$ of the antidominant Weyl chamber (see \cite[Proposition~3.2 and Corollary~4.1,~i)]{BriP} or~\cite[Corollary~5.3]{Kn91}).
We denote this cone by~$\mathcal V_G(X)$.
Results of~\cite[\S\,3]{Bri90} imply that $\mathcal V_G(X)$ is a cosimplicial cone in~$\mathcal Q_G(X)$.
Consequently, there is a uniquely determined linearly independent set $\Sigma_G(X)$ of primitive elements in~$\Lambda_G(X)$ such that
\[
\mathcal V_G(X) = \lbrace q \in \mathcal Q_G(X) \mid q(\sigma) \le 0 \ \text{for all} \ \sigma \in \Sigma_G(X) \rbrace.
\]
Elements of $\Sigma_G(X)$ are called \textit{spherical roots} of~$X$ and $\mathcal V_G(X)$ is called the \textit{valuation cone} of~$X$.
The above discussion implies that every spherical root is a nonnegative linear combination of simple roots.

In this paper, we will need the following important property.

\begin{proposition}[{see~\cite[Corollary~5.3]{BriP}}] \label{prop_finite_index_in_norm}
Let $H \subset G$ be a spherical subgroup.
The set $\Sigma_G(G/H)$ is a basis of the vector space $\QQ\Lambda_G(G/H)$ if and only if the group $N_G(H)/H$ is finite.
\end{proposition}

As can be easily seen from the definitions, the weight lattice and spherical roots depend only on the open $G$-orbit in~$X$.

\begin{remark} \label{remark_G/Z}
If a central subgroup $Z \subset G$ acts trivially on~$X$, then $X$ can be regarded as a spherical $G/Z$-variety.
In this case it is easy to see that the weight lattice and set of spherical roots of~$X$ as a spherical $G$-variety naturally identify with the those of~$X$ as a spherical $G/Z$-variety.
\end{remark}

\subsection{Spherical modules}
\label{subsec_SM_gen}

Given two connected reductive algebraic groups $G_1,G_2$, for $i=1,2$ let $V_i$ be a finite-dimensional $G_i$-module and let $\rho_i \colon G \to \GL(V_i)$ be the corresponding representation.
According to the terminology of Knop~\cite[\S\,5]{Kn98}, the pairs $(G_1,V_1)$ and $(G_2,V_2)$ are said to be \textit{geometrically equivalent} (or just \textit{equivalent} for short) if there exists an isomorphism of vector spaces $V_1 \xrightarrow{\sim} V_2$ identifying the groups $\rho_1(G_1)$ and $\rho_2(G_2)$.
As an important example, note that for any $G$-module $V$ the pairs $(G,V)$ and $(G,V^*)$ are equivalent.

In what follows, let $V$ be a finite-dimensional $G$-module.

Consider a decomposition $V = V_1 \oplus \ldots \oplus V_k$ into a direct sum of simple $G$-modules and let $Z$ be the subgroup of $\GL(V)$ consisting of the elements that act by scalar transformations on each~$V_i$, $i = 1,\ldots,k$.
Let $C \subset Z$ be the image in~$\GL(V)$ of the connected center of~$G$.
We say that $V$ is \textit{saturated} (as a $G$-module) if $C = Z$.
In the general case, one can find a subtorus $C_0 \subset Z$ such that $Z = C \times C_0$, and then $V$ becomes saturated when regarded as a ($G \times C_0$)-module.
The $(G \times C_0)$-module $V$ is called the \textit{saturation} of the $G$-module~$V$.
Note that the pair $(G \times C_0, V)$ is equivalent to $(G' \times Z, V)$.
Up to equivalence, an arbitrary module is obtained from a saturated one by reducing the connected center of the acting group.

We say that $V$ is a \textit{spherical $G$-module} if $V$ is spherical as a $G$-variety.
According to~\cite[Theorem~2]{VK78}, $V$ is spherical if and only if the $G$-module $\KK[V]$ is multiplicity free.
From the latter property (or directly from the definition) it is easily deduced that every submodule of a spherical $G$-module is again spherical and $V$ is spherical if and only if so is~$V^*$.
Observe that the property of $V$ being spherical depends only on the equivalence class of the pair~$(G,V)$.
As follows from~\cite[Proposition~3.5(c)]{Avd_DSS}, passing to the saturation preserves the rank of a spherical module.

A complete classification of simple spherical modules was obtained in \cite{Kac}.
Before discussing the classification of nonsimple spherical modules, we need to introduce several additional notions.

We say that $V$ is \textit{decomposable} if there exist connected reductive algebraic groups $G_1, G_2$, a $G_1$-module~$V_1$, and a $G_2$-module~$V_2$ such that the pair $(G,V)$ is equivalent to $(G_1 \times G_2, V_1 \oplus V_2)$.
Clearly, in this situation $V$ is a spherical $G$-module if and only if $V_i$ is a spherical $G_i$-module for $i=1,2$, in which case one has $\Lambda_G(V) \simeq \Lambda_{G_1}(V_1) \oplus \Lambda_{G_2}(V_2)$.
We say that $V$ is \textit{indecomposable} if $V$ is not decomposable and $V$ is \textit{strictly indecomposable} if the saturation of~$V$ is indecomposable.

A complete classification (up to equivalence) of all strictly indecomposable nonsimple spherical modules was independently obtained in~\cite{BR} and~\cite{Lea}.
A property that is crucial for the present paper is that every such module is the direct sum of at most two simple modules.
Both papers~\cite{BR} and~\cite{Lea} contain also a description of all spherical modules with a given saturation, which completes the classification of all spherical modules.
A complete list (up to equivalence) of all indecomposable saturated spherical modules can be found in~\cite[\S\,5]{Kn98} along with various additional data.
Among these data, we will need in this paper the values of the rank.

\subsection{A reduction for spherical modules}
\label{subsec_char_of_SM}

Let $V$ be a finite-dimensional $G$-module (not necessarily simple) and let $\omega$ be a highest weight of~$V$.
Fix a highest-weight vector $v_\omega \in V$ of weight~$\omega$.
Put $Q = \lbrace g \in G \mid g\langle v_\omega \rangle = \langle v_\omega \rangle \rbrace$; this is a parabolic subgroup of~$G$ containing~$B$.
Let $M$ be the standard Levi subgroup of~$Q$ and let $M_0$ be the stabilizer of~$v_\omega$ in~$M$.
Let $Q^- \supset B^-$ be the parabolic subgroup of~$G$ opposite to $Q$.
Fix a lowest weight vector $\xi \in V^*$ of weight~$-\omega$, so that $\xi(v_\omega) \ne 0$.
Put
\[
\widetilde V = \lbrace v \in V \mid (\mathfrak q_u \xi)(v) = 0 \rbrace = \lbrace v \in V \mid \xi(\mathfrak q_u v) = 0 \rbrace
\]
and $V_0 = \widetilde V \cap \Ker \xi$.
Both $V$ and $V_0$ are $M$-modules in a natural way, and there are the decompositions $\widetilde V = \langle v_\omega \rangle \oplus V_0$ and $V = (\mathfrak q^-_u v_\omega) \oplus \widetilde V$ into direct sums of $M$-submodules.
The following result is extracted from the proof of \cite[Theorem~3.3]{Kn98}.

\begin{theorem} \label{thm_sph_modules_imp}
The following assertions hold:
\begin{enumerate}[label=\textup{(\alph*)},ref=\textup{\alph*}]
\item \label{thm_sph_modules_imp_a}
$V$ is a spherical $G$-module if and only if $\widetilde V$ is a spherical $M$-module.

\item \label{thm_sph_modules_imp_b}
Under the conditions of~\textup(\ref{thm_sph_modules_imp_a}\textup), one has $\Lambda_G(V^*) = \Lambda_M(\widetilde V^*)$.
\end{enumerate}
\end{theorem}

The above theorem implies the following algorithm for determining the sphericity of an $L$-module~$V$ for a standard Levi subgroup~$L \subset G$; see~\cite[Theorem~3.3 and the paragraph preceding it]{Kn98}.
We denote by $\Omega_L$ the multiset (that is, multiplicities are allowed) of $T$-weights of~$V$.

\medskip

Algorithm~\newalg: \label{alg_sph_mod}

Input: a triple $(\Pi_L$, $\Delta^+_L$, $\Omega_L$)

Step~\step:
choose $\omega \in \Omega$ such that $(\omega + \Pi_L) \cap \Omega_L = \varnothing$;

Step~\step:
compute the set $\Pi_M = \lbrace \alpha \in \Pi_L \mid \langle \alpha^\vee, \omega \rangle = 0 \rbrace$;

Step~\step:
compute the sets $\Delta^+_M = \lbrace \alpha \in \Delta^+_L \mid \langle \alpha^\vee, \omega \rangle = 0 \rbrace$ and $\Delta^+_L \setminus \Delta^+_M = \lbrace \alpha \in \Delta^+_L \mid \langle \alpha^\vee, \omega \rangle > 0 \rbrace$;

Step~\step:
compute the set $\Omega_M = \Omega_L \setminus (\lbrace \omega \rbrace \cup \lbrace \omega - \alpha \mid \alpha \in \Delta^+_L \setminus \Delta^+_M \rbrace)$;

Step~\step:
if $\Omega_M = \varnothing$, then return $\lbrace \omega \rbrace$;

Step~\step:
if $\Omega_M \ne \varnothing$, then return $\lbrace \omega \rbrace \cup [\text{output of the algorithm for $(\Pi_M, \Delta^+_M, \Omega_M)$}]$.

\medskip

Any output of Algorithm~\ref{alg_sph_mod} is a multiset $\Theta$ of $T$-weights. 

\begin{proposition} \label{prop_sph_lin_ind}
$V$ is a spherical $L$-module if and only if $\Theta$ is linearly independent.
\end{proposition}

\subsection{Some results on spherical subgroups}
\label{subsec_gen_SS}

Given an arbitrary subgroup $H \subset G$,
by \cite[\S\,30.3]{Hum} there exists a parabolic subgroup $P \subset G$ such that $H \subset P$ and $H_u \subset P_u$.
In this situation, we say that $H$ is \textit{regularly embedded} in~$P$.
One can choose Levi subgroups $L \subset P$ and $K \subset H$ in such a way that $K \subset L$.
Then by \cite[Lemma~1.4]{Mon} there is a $K$-equivariant isomorphism $P_u/H_u \simeq \mathfrak p_u / \mathfrak h_u$.

Replacing $H$, $P$, and $L$ with conjugate subgroups, we may assume that $P \supset B^-$ and $L$ is the standard Levi subgroup of~$P$.

In this paper we will deal with subgroups $H$ satisfying $K=L$.
For such subgroups, there is the following result, which is implied by~\cite[Proposition~I.1]{Br87} and~\cite[Theorem~1.2]{Pa94}; see also~\cite[Theorem~9.4]{Tim}.

\begin{proposition} \label{prop_criterion_spherical}
Under the above notation and assumptions suppose in addition that $K = L$.
Then the following conditions are equivalent.
\begin{enumerate}[label=\textup{(\arabic*)},ref=\textup{\arabic*}]
\item
$H$ is a spherical subgroup of~$G$.

\item
$\mathfrak p_u/\mathfrak h_u$ is a spherical $L$-module.
\end{enumerate}
Moreover, if these conditions hold, then $\Lambda_G(G/H) = \Lambda_L(\mathfrak p_u/\mathfrak h_u)$ where the lattice $\Lambda_L(\mathfrak p_u/\mathfrak h_u)$ is taken with respect to~$B_L$.
\end{proposition}

\begin{remark}
Under the conditions of Proposition~\ref{prop_criterion_spherical}, some partial results on the set of spherical roots of $G/H$ were obtained in~\cite{Pez}.
Namely, Corollary~8.12 and Theorem~6.15 in loc. cit. assert that
\[
\Sigma_L(\mathfrak p_u / \mathfrak h_u) = \lbrace \sigma \in \Sigma_G(G/H) \mid \Supp \sigma \subset \Pi_L \rbrace \ \: \text{and} \ \: \Sigma_G(G/H) \setminus \Sigma_L(\mathfrak p_u / \mathfrak h_u) \subset \Delta^+ \setminus \Delta^+_L,
\]
respectively.
In particular, in the situation of Proposition~\ref{prop_criterion_spherical}  each spherical root of the spherical $L$-module $\mathfrak p_u / \mathfrak h_u$ is automatically a spherical root of $G/H$.
Since the spherical roots of all spherical modules are known from~\cite[\S\,5]{Kn98}, in this way one may obtain all spherical roots of $G/H$ whose support is contained in~$\Pi_L$.
\end{remark}

\section{Degenerations and algorithms for computing the spherical roots}
\label{sect_active_C-roots}

\subsection{Description of the setting}
\label{ssec_setting}

In this subsection, we fix the setting and notation that will be used throughout the whole section.

Let $P \subset G$ be a parabolic subgroup such that $P \supset B^-$ and let $L$ be the standard Levi subgroup of~$P$.
Let $P^+ \supset B$ be the parabolic subgroup of~$G$ opposite to~$P$.
Denote by~$C$ the connected center of~$L$ and retain all the notation and terminology of~\S\,\ref{subsec_Levi_roots}.

We introduce the following additional notation:
\begin{itemize}
\item
for every $\lambda \in \Phi$, the symbol $\widehat \lambda$ stands for the highest weight of the $L$-module $\mathfrak g(\lambda)$;

\item
for every $\delta \in \Delta$, the symbol $\overline \delta$ denotes the image of~$\delta$ under the restriction map $\mathfrak X(T) \to \mathfrak X(C)$.
\end{itemize}

Suppose that $H \subset G$ is a subgroup (not necessarily spherical) regularly embedded in~$P$ and such that $L$ is a Levi subgroup of~$H$, so that $H = L \rightthreetimes H_u$.
Put
\[
\Psi = \Psi(H) = \lbrace \mu \in \Phi^+ \mid \mathfrak g(-\mu) \not \subset \mathfrak h \rbrace.
\]
According to~\cite[Definition~4.1]{Avd_DSS}, elements of $\Psi$ are called \textit{active $C$-roots} of~$H$. (Note that this notion is well defined without the sphericity assumption for~$H$.)
The following property of~$\Psi$ is readily implied by Proposition~\ref{prop_properties_of_g(lambda)}(\ref{prop_properties_of_g(lambda)_b}).

\begin{lemma} \label{lemma_sum}
If $\lambda \in \Psi$ and $\lambda = \mu + \nu$ for some $\mu, \nu \in \Phi^+$, then $\lbrace \mu, \nu \rbrace \cap \Psi \ne \varnothing$.
\end{lemma}

It will be convenient for us to work with the subspace~$\mathfrak h^\perp \subset \mathfrak g$.
We have
\begin{equation} \label{eqn_h_perpII}
\mathfrak h^\perp = \mathfrak p_u \oplus \bigoplus \limits_{\mu \in \Psi} \mathfrak g(\mu).
\end{equation}
Put $\mathfrak u = \mathfrak h^\perp \cap \mathfrak p_u^+$ for short; then  $\mathfrak u = \bigoplus \limits_{\mu \in \Psi} \mathfrak g(\mu)$.
Note that $\mathfrak u$ is an $L$-module in a natural way and by Proposition~\ref{prop_properties_of_g(lambda)}(\ref{prop_properties_of_g(lambda)_c}) there is a natural $L$-module isomorphism $\mathfrak u \simeq (\mathfrak p_u / \mathfrak h_u)^*$.

Recall from Proposition~\ref{prop_criterion_spherical} that $H$ being a spherical subgroup of~$G$ is equivalent to $\mathfrak p_u / \mathfrak h_u$ (and hence $\mathfrak u$) being a spherical $L$-module.
In this case one has $N_G(H)^0 = H$ by~\cite[Proposition~3.23]{Avd_DSS}, so Propositions~\ref{prop_finite_index_in_norm} and~\ref{prop_criterion_spherical} imply the following result.

\begin{proposition} \label{prop_num_sr}
Suppose that $H$ is spherical.
Then $|\Sigma_G(G/H)| = \rk_G(\mathfrak p_u / \mathfrak h_u) = \rk_G(\mathfrak u)$.
\end{proposition}

As $H$ contains the center of~$G$, for every central subgroup $Z \subset G$ one has $G/H \simeq (G/Z) / (H/Z)$ as $G/Z$-varieties.
Thanks to Remark~\ref{remark_G/Z}, this means that for computing the set $\Sigma_G(G/H)$ it suffices to restrict ourselves to the case of semisimple~$G$.

In the next subsections (\S\S\,\ref{subsec_reduction_AG}--\ref{ssec_opt_alg}), we assume that $G$ is semisimple and $H$ is spherical.

\subsection{Reduction of the ambient group}
\label{subsec_reduction_AG}

In this subsection we recall from~\cite[\S\,5.3]{Avd_DSS} a natural reduction that under certain conditions enables one to pass from the pair $(G,H)$ to another pair $(G_0,H_0)$ with a ``smaller'' group~$G_0$.
This reduction keeps all the combinatorics of active roots unchanged and preserves the set of spherical roots.
It can be applied before any step of all algorithms discussed in this section.

Consider the set $\Pi_0 = \bigcup \limits_{\lambda \in \Psi} \Supp \widehat \lambda$ and let $L_0 \subset G$ be the standard Levi subgroup with $\Pi_{L_0} = \Pi_0$.
Put $G_0 = L_0'$ and $H_0 = G_0 \cap H$.
Then it is easy to see that $H_0$ is regularly embedded in the parabolic subgroup $P \cap G_0 \subset G_0$ with standard Levi subgroup $L \cap G_0$, which is also a Levi subgroup of~$H_0$.
The connected center of $L \cap G_0$ equals $C \cap G_0$ and $(L \cap G_0)'$ coincides with the product of all simple factors of $L'$ contained in~$G_0$.
If $\mathfrak g(\lambda) \cap \mathfrak g_0 \ne \lbrace 0 \rbrace$ for some $\lambda \in \Phi$, then $\mathfrak g(\lambda) \subset \mathfrak g_0$ and $\mathfrak g(\lambda)$ is simple as an $(L \cap G_0)$-module.
It follows that the objects $\Psi$ and $\mathfrak u$ are naturally identified with those for~$H_0$ and the pairs $(L, \mathfrak h_0)$, $(L,\mathfrak u)$ are equivalent to $(L \cap G_0, \mathfrak h_0)$, $(L \cap G_0, \mathfrak u)$, respectively.

We say that the pair $(G_0, H_0)$ is obtained from $(G,H)$ by \textit{reduction of the ambient group}.
In the next statement, the set of simple roots of~$G_0$, which is~$\Pi_0$, is regarded as a subset of~$\Pi$.

\begin{proposition}[{\cite[Proposition~5.3]{Avd_DSS}}] \label{prop_reduction_AG_sph_roots}
One has $\Sigma_G(G/H) = \Sigma_{G_0}(G_0/H_0)$.
\end{proposition}

\subsection{Construction of degenerations and the base algorithm}
\label{subsec_degen_add}

In this subsection we present the construction of degenerations (called additive degenerations in~\cite{Avd_DSS}), which is the most important part of our algorithms for computing the set $\Sigma_G(G/H)$.
This construction depends on the choice of $\lambda \in \Psi$, which is assumed to be fixed throughout this subsection.

Put $\delta = \widehat \lambda$ and let $\mathfrak s(\delta) \simeq \mathfrak{sl}_2$ be the subalgebra of~$\mathfrak g$ spanned by $e_\delta$, $h_\delta$, and $e_{-\delta}$.
Consider the one-parameter unipotent subgroup $\phi \colon \GG_a \to G$ given by $\phi(t) = \exp(te_{-\delta})$.
For every $t \in \GG_a$, we put $\mathfrak h_t = \phi(t)\mathfrak h$.
According to~\cite[Proposition~2.4]{Avd_DSS}, there exists $\lim \limits_{t \to \infty} \mathfrak h_t$; we denote it by~$\mathfrak h_\infty$.
In what follows, $\mathfrak h_\infty$ is referred to as the degeneration of~$\mathfrak h$ defined by~$\lambda$.

Note that $\mathfrak h_t^\perp = \phi(t) \mathfrak h^\perp$ and $\mathfrak h_\infty^\perp = \lim \limits_{t \to \infty} \mathfrak h_t^\perp$.

To describe the subalgebra $\mathfrak h_\infty$, we introduce the set \[
Y(\delta) = \lbrace \alpha \in \Delta \mid \alpha+ \delta \notin \Delta \rbrace.
\]
For every $\alpha \in Y(\delta)$, let $V(\alpha) \subset \mathfrak g$ be the $\mathfrak s(\delta)$-submodule generated by~$e_\alpha$.
The following properties of $V(\alpha)$ are straightforward:
\begin{itemize}
\item
$V(\alpha)$ is a simple $\mathfrak s(\delta)$-module with highest weight $\delta^\vee(\alpha)$;

\item
$e_\alpha$ is a highest-weight vector of~$V(\alpha)$;

\item
$V(\alpha)$ is $T$-stable.
\end{itemize}
Then there is the following decomposition of $\mathfrak g$ into a direct sum of $\mathfrak s(\delta)$-submodules:
\begin{equation} \label{eqn_g_decompII}
\mathfrak g = (h_\delta^\perp \cap \mathfrak t) \oplus \bigoplus \limits_{\alpha \in Y(\delta)} V(\alpha).
\end{equation}
Comparing this with~(\ref{eqn_h_perpII}) we find that
\begin{equation} \label{eqn_h^perp_decompII}
\mathfrak h^\perp = \bigoplus \limits_{\alpha \in Y(\delta)} (\mathfrak h^\perp \cap V(\alpha)).
\end{equation}
By Proposition~\ref{prop_limitII_prelim}, for every $\alpha \in Y(\delta)$ there exists $\lim \limits_{t \to \infty} (\mathfrak h^\perp_t \cap V(\alpha))$, which we will denote by $(\mathfrak h^\perp \cap V(\alpha))_\infty$.
Then decompositions~(\ref{eqn_g_decompII}) and~(\ref{eqn_h^perp_decompII}) imply the decomposition
\begin{equation} \label{eqn_h_perp_lim}
\mathfrak h_\infty^\perp = \bigoplus \limits_{\alpha \in Y(\delta)} (\mathfrak h^\perp \cap V(\alpha))_\infty.
\end{equation}

For every $\alpha \in Y(\delta)$ the limit $(\mathfrak h^\perp \cap V(\alpha))_\infty$ is determined using Proposition~\ref{prop_limitII_prelim}.
Since the subspace $\mathfrak h^\perp \cap V(\alpha) \subset V(\alpha)$ is $h_\delta$-stable, $(\mathfrak h^\perp \cap V(\alpha))_\infty$ is described in terms of shifting $h_\delta$-weight subspaces in $\mathfrak h^\perp \cap V(\alpha)$ as explained in the paragraph after Proposition~\ref{prop_limitII_prelim}.

To state the main properties of $\mathfrak h_\infty$, we apply the construction of \S\,\ref{subsec_char_of_SM} with $G = L$, $V = \mathfrak u$, and $\omega = \delta$.
Put $Q = \lbrace g \in L \mid \Ad(g) \mathfrak g_{\delta} = \mathfrak g_{\delta} \rbrace$; this is a parabolic subgroup of~$L$ containing $B_L$.
Let $Q^- \supset B^- \cap L$ be the parabolic subgroup of~$L$ opposite to~$Q$.
Let $M$ be the standard Levi subgroup of~$Q$ and let $M_0$ be the stabilizer of $e_\delta$ in~$M$.
Put $\Delta^+_L(\delta) = \lbrace \alpha \in \Delta_L^+ \mid (\alpha, \delta) > 0 \rbrace$, so that $\Delta^+_L(\delta) = \Delta^+_L \setminus \Delta^+_M$.
Regard the element $e_{-\delta}$ as a linear function on~$\mathfrak u$ via the fixed $G$-invariant inner product on~$\mathfrak g$.
Put
\[
\widetilde {\mathfrak u} = \lbrace x \in \mathfrak u \mid (\mathfrak q_u e_{-\delta})(x) = 0 \rbrace
\]
and $\mathfrak u_0 = \widetilde {\mathfrak u} \cap \Ker e_{-\delta}$.
Note that
\[
\mathfrak u_0 = \bigoplus \limits_{\substack{\alpha \in \Delta^+ \colon \overline \alpha \in \Psi, \\ \delta - \alpha \notin \Delta^+_L(\delta) \cup \lbrace 0 \rbrace}} \mathfrak g_{\alpha}.
\]
Then there is the decomposition $\mathfrak u = \mathfrak g_{\delta} \oplus [\mathfrak q^-_u, \mathfrak g_\delta] \oplus \mathfrak u_0$ into a direct sum of $M$-modules.
Consider the decomposition $\mathfrak h_\infty^\perp = (\mathfrak h_\infty^\perp \cap \mathfrak p_u) \oplus (\mathfrak h_\infty^\perp \cap \mathfrak l) \oplus (\mathfrak h_\infty^\perp \cap \mathfrak p^+_u)$ and put $\mathfrak u_\infty = \mathfrak h_\infty^\perp \cap \mathfrak p_u^+$ for short.

\begin{proposition}[{\cite[Proposition~5.11]{Avd_DSS}}] \label{prop_h0_typeII}
The following assertions hold.
\begin{enumerate}[label=\textup{(\alph*)},ref=\textup{\alph*}]
\item \label{prop_h0_typeII_a}
$\mathfrak h_\infty^\perp \cap \mathfrak p_u = \mathfrak p_u$.

\item \label{prop_h0_typeII_b}
$\mathfrak h_\infty^\perp \cap \mathfrak l = \mathfrak q_u^- \oplus \langle h_\delta \rangle$.

\item \label{prop_h0_typeII_c}
The subspace $\mathfrak u_\infty$ is $M$-stable and there is an $M_0$-module isomorphism $\mathfrak u_0 \simeq \mathfrak u_\infty$. Moreover, under this isomorphism each highest-weight vector in $\mathfrak u_0$ of $T$-weight~$\alpha$ corresponds to a highest-weight vector in $\mathfrak u_\infty$ of $T$-weight $\alpha - k_\alpha \delta$ for some $k_\alpha \in \ZZ_{\ge0}$.
\end{enumerate}
\end{proposition}

Put $R = Q^- \rightthreetimes P_u$; then $R$ is a standard parabolic subgroup of~$G$ containing $B^-$ and having $M$ as a Levi subgroup.
Let $H_\infty \subset G$ be the connected subgroup with Lie algebra~$\mathfrak h_\infty$ and consider the subgroup $N = N(\lambda) = N_G(H_\infty)^0$.
We note that $H_\infty$ is automatically spherical in~$G$ by \cite[Proposition~1.3(i)]{Bri90}, hence $N$ is also spherical.
In the next theorem, parts (\ref{thm_degen_a},\,\ref{thm_degen_b}) are just \cite[Theorem~5.12(a,\,b)]{Avd_DSS} while part~(\ref{thm_degen_c}) follows from \cite[Theorem~5.12(c,\,d)]{Avd_DSS} and the discussion in~\cite[\S\,3.9]{Avd_DSS}.

\begin{theorem} \label{thm_degen}
The following assertions hold.
\begin{enumerate}[label=\textup{(\alph*)},ref=\textup{\alph*}]
\item \label{thm_degen_a}
The subgroup $H_\infty$ is regularly embedded in $R$, $M_0$ is a Levi subgroup of~$H_\infty$, and $\mathfrak r_u / (\mathfrak h_\infty)_u \simeq \mathfrak u_\infty^*$ as $M_0$-modules.

\item \label{thm_degen_b}
$N = M \rightthreetimes (H_\infty)_u$.
In particular, $\dim N = \dim H + 1$.

\item \label{thm_degen_c}
There is $\sigma \in \Sigma_G(G/H)$ such that $\Sigma_G(G/H) \setminus \lbrace \sigma \rbrace = \Sigma_G(G/N)$.
\end{enumerate}
\end{theorem}

Now suppose $\lambda_1,\lambda_2 \in \Psi$ and $\lambda_1 \ne \lambda_2$.
For $i=1,2$ put $N_i = N(\lambda_i)$ for short and let $\sigma_i \in \Sigma_G(G/H)$ be the spherical root appearing in Theorem~\ref{thm_degen}(\ref{thm_degen_c}), so that $\Sigma_G(G/H) \setminus \lbrace \sigma_i \rbrace = \Sigma_G(G/N_i)$.
The next result is implied by~\cite[Proposition~5.13]{Avd_DSS}.

\begin{proposition} \label{prop_sr_union}
One has $\sigma_1 \ne \sigma_2$, so that $\Sigma_G(G/H) = \Sigma_G(G/N_1) \cup \Sigma_G(G/N_2)$.
\end{proposition}

All subgroups $H$ satisfying $|\Psi| = 1$ were listed in \cite[\S\,5.6]{Avd_DSS} along with the corresponding sets $\Sigma_G(G/H)$.
We reproduce this information in Theorem~\ref{thm_par1psi1}.

Below we describe the base algorithm, which reduces computing the set $\Sigma_G(G/H)$ to the same problem for several subgroups $H_1,\ldots,H_k$ satisfying $|\Psi(H_1)| = \ldots = |\Psi(H_k)| = 1$.

\medskip

Algorithm~\newalg: \label{alg_ba}

Input: $H$

Step~\step:
if $|\Psi| = 1$, then exit and return~$H$;

Step~\step:
choose $\lambda_1 \in \Psi$ and compute the subgroup~$N(\lambda_1)$;

Step~\step:
choose $\lambda_2 \in \Psi \setminus \lbrace \lambda_1 \rbrace$ and compute the subgroup~$N(\lambda_2)$;

Step~\step:
repeat the procedure for~$N(\lambda_1)$;

Step~\step:
repeat the procedure for~$N(\lambda_2)$.

\medskip

Any output of Algorithm~\ref{alg_ba} is a collection of spherical subgroups $H_1,\ldots,H_k \subset G$ with $|\Psi(H_1)| = \ldots = |\Psi(H_k)| = 1$, which satisfies $\Sigma_G(G/H) = \Sigma_G(G/H_1) \cup \ldots \cup \Sigma_G(G/H_k)$ by Proposition~\ref{prop_sr_union}.

\subsection{The SM-decomposition}

There is a decomposition into a disjoint union
\begin{equation} \label{eqn_SM-dec}
\Psi = \Psi_1 \cup \ldots \cup \Psi_p
\end{equation}
with the following properties:
\begin{itemize}
\item
for every simple factor $F$ of $L'$ acting nontrivially on $\mathfrak u$ there exists a unique $i \in \lbrace 1,\ldots, p \rbrace$ such that $F$ acts trivially on each $\mathfrak g(\mu)$ with $\mu \notin \Psi_i$;

\item
for every $i = 1,\ldots,p$, the saturation of the $L$-module $\mathfrak u^i = \bigoplus \limits_{\mu \in \Psi_i} \mathfrak g(\mu)$ is indecomposable (equivalently, $\mathfrak u^i$ is indecomposable as an $L'$-module).
\end{itemize}
Following~\cite[\S\,6.1]{Avd_DSS}, we call decomposition~(\ref{eqn_SM-dec}) the \textit{SM-decomposition} of~$\Psi$.
Note that the components of this decomposition are uniquely determined up to permutation.
We say that the SM-decomposition of~$\Psi$ is \textit{trivial} if it has exactly one component.

For every $i = 1,\ldots, p$, let $S_i$ be the product of simple factors of $L'$ that act nontrivially on~$\mathfrak u^i$.

Fix $i \in \lbrace 1,\ldots, p \rbrace$ and $\lambda \in \Psi \setminus \Psi_i$.
Put $\delta = \widehat \lambda$, apply the degeneration construction of~\S\,\ref{subsec_degen_add} for~$\lambda$, and consider the resulting subgroup $N = N(\lambda)$.
In what follows, the analogues for $N$ of objects like $\Psi, \mathfrak u, \ldots$ defined for $H$ will be denoted like $\Psi(N), \mathfrak u(N), \ldots$.

Let $\Psi(N) = \Psi_1(N) \cup \ldots \cup \Psi_q(N)$ be the SM-decomposition of~$\Psi(N)$.
In the next proposition, by abuse of notation, $\mathfrak g(\mu_*)$ and $\widehat \mu_*$ denote the corresponding objects defined for~$N$.

\begin{proposition}[{\cite[Proposition~6.3]{Avd_DSS}}] \label{prop_unique_j}
There exists a unique $j \in \lbrace 1,\ldots,q \rbrace$ with the following properties:
\begin{enumerate}[label=\textup{(\arabic*)},ref=\textup{\arabic*}]
\item
$\mathfrak u^i$ shifts to $\mathfrak u^j(N)$ under the degeneration;

\item
$S_j(N) = S_i$;

\item
there is a bijection $\Psi_i \to \Psi_j(N)$, $\mu \mapsto \mu_*$, such that for every $\mu \in \Psi_i$ one has an $S_i$-module isomorphism $\mathfrak g(\mu_*) \simeq \mathfrak g(\mu)$ and $\widehat \mu_* = \widehat \mu - c_\mu \delta$ where $c_\mu \ge 0$.
\end{enumerate}
\end{proposition}

\subsection{The optimized algorithm}
\label{ssec_opt_alg}

Let $\Psi = \Psi_1 \cup \ldots \cup \Psi_p$ be the SM-decomposition of~$\Psi$.
The optimized algorithm applies when $p \ge 2$ and rests on the following idea.
Each spherical root of $G/H$ is somehow ``controlled'' by exactly one component of the SM-decomposition of~$\Psi$, and to ``extract'' all spherical roots controlled by a given component~$\Psi_i$ we perform a modification (possibly involving degenerations) of~$H$ to obtain a new spherical subgroup $H_i$ such that the pair $(L,\mathfrak u^i)$ is equivalent to $(L(H_i), \mathfrak u(H_i))$ and the spherical roots of $H_i$ are precisely those of $G/H$ controlled by~$\Psi_i$.
In this way, we obtain a fast algorithm that reduces computing the spherical roots for $H$ to the same problem for several other spherical subgroups for which the SM-decomposition is trivial.

For every $i = 1,\ldots, p$, put
$\Upsilon_i= \lbrace \mu \in \Psi \setminus \Psi_i \mid \Supp \mu \subset \Supp(\Psi_i) \rbrace$ and observe that $\mathfrak l \oplus \bigoplus \limits_{\mu \in \Phi^+ \setminus(\Psi_i \cup \Upsilon_i)} \mathfrak g(-\mu)$ is a subalgebra of~$\mathfrak g$.
Let $\widehat H_i \subset G$ be the corresponding connected subgroup.
Note that $H \subset \widehat H_i$, $L$ is a Levi subgroup of~$\widehat H_i$, $\Psi(\widehat H_i) = \Psi_i \cup \Upsilon_i$, and $\Psi_i$ is a component of the SM-decomposition of $\Psi(\widehat H_i)$.

Given a subset $\Theta \subset \Psi$ and an element $\nu \in \Theta$, we say that $\nu$ is an \textit{upper} element of~$\Theta$ if $\mu - \nu \notin \Phi^+$ for all $\mu \in \Theta \setminus \lbrace \nu \rbrace$.
Observe that every nonempty subset of $\Psi$ contains at least one upper element.

Here is the description of~\cite[Algorithm~D]{Avd_DSS}.

\medskip

Algorithm~\newalg: \label{alg_H_i}

Input: a pair $(H,\Psi_i)$

Step~\step: \label{step_D1}
replace $(H, \Psi_i)$ with $(\widehat H_i, \Psi_i)$;

Step~\step: \label{step_A1}
if $\Upsilon_i = \varnothing$, then exit and return $H$;

Step~\step: \label{step_A2}
choose an upper element $\lambda \in \Upsilon_i$ and compute the subgroup $N = N(\lambda)$;

Step~\step: \label{step_A3}
identify $j$ as in Proposition~\ref{prop_unique_j};

Step~\step: \label{step_D3}
repeat the procedure for the pair $(N, \Psi_j(N))$.

\medskip

For $i = 1,\ldots,p$, let $H_i$ be an output of Algorithm~\ref{alg_H_i} for the pair $(H, \Psi_i)$.
Then $H_i$ is regularly embedded in a parabolic subgroup $P_i \subset G$ such that $P_i \supset B^-$.
Let $L_i$ be the standard Levi subgroup of~$P_i$, which is simultaneously a Levi subgroup of $H_i$.
Then it follows from the description of the algorithm along with Proposition~\ref{prop_unique_j} that the SM-decomposition of $\Psi(H_i)$ is trivial and the pair $(L_i', \mathfrak u(H_i))$ is equivalent to $(L', \mathfrak u^i)$.
The next theorem follows from \cite[Theorem~6.14 and Proposition~6.15]{Avd_DSS}.

\begin{theorem} \label{thm_opt_reduction}
There is a disjoint union $\Sigma_G(G/H) = \Sigma_G(G/H_1) \cup \ldots \cup \Sigma_G(G/H_p)$.
\end{theorem}

\begin{remark}
Being an output of Algorithm~\ref{alg_H_i}, the subgroup~$H_i$ depends on the sequence of choices of~$\lambda$ at each execution of step~\ref{step_A2}.
We conjecture that the pair $(G_i,K_i)$ obtained from $(G,H_i)$ by reduction of the ambient group does not depend of this sequence of choices and hence is uniquely determined by the pair $(H, \Psi_i)$.
\end{remark}

\section{Classification of cases with trivial SM-decomposition}
\label{sect_comp_SR}

Retain the setting and notation of~\S\,\ref{ssec_setting}.
In this section, we present the classification of all cases where $H$ is spherical and the SM-decomposition of the set $\Psi$ is trivial.
Recall that in this case one has $|\Psi| \le 2$ by the classification of spherical modules.
Thanks to the reduction of the ambient group (see \S\,\ref{subsec_reduction_AG}), we may restrict ourselves to the case $\Supp \Psi = \Pi$, which will be assumed throughout.

\subsection{Preliminary results}
\label{subsec_primitive_cases}

In this subsection we state several necessary conditions implied by the fact that the SM-decomposition of~$\Psi$ is trivial.
In particular, we find that $G$ is necessarily simple.

For every $\lambda \in \Psi$, let $\widetilde \lambda$ denote the lowest weight of the $L$-module $\mathfrak g(\lambda)$.

First recall the following well-known lemma.

\begin{lemma} \label{lemma_aux}
For every $\alpha \in \Delta^+ \setminus \Pi$ there exists $\beta \in \Pi$ such that $(\alpha,\beta) > 0$.
In particular, $\alpha - \beta \in \Delta^+$.
\end{lemma}

\begin{lemma} \label{lemma_alpha_in_Pi}
Suppose that $\Psi \ne \varnothing$.
Then there exists $\alpha \in \Pi \setminus \Pi_L$ such that $\overline \alpha \in \Psi$.
\end{lemma}

\begin{proof}
Choose $\lambda \in \Psi$ with minimal possible~$\operatorname{ht} \widetilde \lambda$ and put $\alpha = \widetilde \lambda$ for short.
Assume that $\alpha \notin \Pi$.
Then by Lemma~\ref{lemma_aux} there is $\beta \in \Pi$ such that $\gamma = \alpha - \beta \in \Delta^+$.
As $\alpha$ is the lowest weight of~$\mathfrak g(\lambda)$, one has $\Supp \beta \not\subset \Pi_L$, hence $\overline \beta, \overline\gamma \ne 0$.
Then by Lemma~\ref{lemma_sum} there is $\mu \in \lbrace \overline \beta, \overline \gamma \rbrace$ such that $\mu \in \Psi$.
Clearly, $\operatorname{ht} \widetilde \mu < \operatorname{ht} \alpha$, a contradiction.
Thus $\alpha \in \Pi$; moreover, $\alpha \in \Pi \setminus \Pi_L$ since $\overline \alpha = \lambda \ne 0$.
\end{proof}

The next result readily follows from Lemma~\ref{lemma_alpha_in_Pi}.

\begin{proposition} \label{prop_|Psi|=1}
Suppose that $H$ is spherical and $|\Psi| = 1$.
Then the following assertions hold.
\begin{enumerate}[label=\textup{(\alph*)},ref=\textup{\alph*}]
\item
$G$ is simple.

\item
There is $\alpha \in \Pi$ such that $\Pi \setminus \Pi_L = \lbrace \alpha \rbrace$.

\item
$\Psi = \lbrace \overline \alpha \rbrace$.
\end{enumerate}
\end{proposition}

\begin{proposition} \label{prop_|Psi|=2}
Suppose that $H$ is spherical and $|\Psi| = 2$.
Then the following assertions hold.
\begin{enumerate}[label=\textup{(\alph*)},ref=\textup{\alph*}]
\item \label{prop_|Psi|=2_a}
If the SM-decomposition of~$\Psi$ is trivial, then $G$ is simple.

\item \label{prop_|Psi|=2_b}
$|\Pi \setminus \Pi_L| \le 2$.

\item \label{prop_|Psi|=2_c}
There is $\alpha \in \Pi \setminus \Pi_L$ such that $\overline \alpha \in \Psi$.
\end{enumerate}
\end{proposition}

\begin{proof}
Since $\mathfrak u$ is a strictly indecomposable $L$-module and every simple factor of $L'$ can act nontrivially only on one simple ideal of~$\mathfrak g$, we get~(\ref{prop_|Psi|=2_a}).
Part~(\ref{prop_|Psi|=2_c}) follows from Lemma~\ref{lemma_alpha_in_Pi}.
It remains to prove~(\ref{prop_|Psi|=2_b}).
Put $\Pi_0 = \Pi \setminus (\Pi_L \cup \alpha)$ and assume that $|\Pi_0| \ge 2$.
Let $\lambda$ be the unique element of $\Psi \setminus \lbrace \overline \alpha \rbrace$ and put $\delta_1 = \widetilde \lambda$.
Since $\Supp \Psi = \Pi$, we have $\Pi_0 \subset \Supp \delta_1$.
Thanks to Lemma~\ref{lemma_aux}, there is $\beta_1 \in \Pi$ such that $(\delta_1, \beta_1) > 0$.
Let $r_1$ be the reflection corresponding to~$\beta_1$ and put $\delta_2 = r_1(\delta_1) = \delta_1 - \langle \beta_1^\vee, \delta_1\rangle \beta_1$.
Then $\Supp \delta_1 \setminus \lbrace \beta_1 \rbrace \subset \Supp \delta_2 \subset \Supp \delta_1$.
Iterating this procedure we construct a sequence $\delta_1, \ldots, \delta_m \in \Delta^+$ along with a sequence $\beta_1,\ldots,\beta_{m-1} \in \Pi$ such that $\delta_m \in \Pi$ and for all $i = 1,\ldots,m-1$ one has $(\delta_i, \beta_i) > 0$, $\delta_i - \beta_i \in \Delta^+$, and $\delta_{i+1} = r_i(\delta_i)$ where $r_i$ is the reflection corresponding to~$\beta_i$.
Clearly, there exists a minimal $k$ such that $\beta_k \in \Pi_0$.
Since $\delta_k - \beta_k \in \Delta^+$, we have $\delta_1 = \beta + \gamma$ where $\beta, \gamma \in \Delta^+$, $\beta = r_1(\ldots r_{k-1}(\beta_k)), \gamma = r_1(\ldots r_{k-1}(\delta_k - \beta_k))$ for $k \ge 2$ and $\beta = \beta_1, \gamma = \delta_1 - \beta_1$ for $k=1$.
Observe that $\Supp \beta \cap \Pi_0 = \lbrace \beta_k \rbrace$ and $\Pi_0 \setminus \lbrace \beta_k \rbrace \subset \Supp \gamma$, which implies $\overline \alpha \notin \lbrace \overline \beta, \overline \gamma \rbrace$.
Since $\delta_1$ is the lowest weight vector in~$\mathfrak g(\lambda)$, we have $\lambda \notin \lbrace \overline \beta, \overline \gamma \rbrace$.
On the other hand, Lemma~\ref{lemma_sum} yields $\lbrace \overline \beta, \overline \gamma \rbrace \cap \Psi \ne \varnothing$, a contradiction.
Thus $|\Pi_0| \le 1$ and the proof of part~(\ref{prop_|Psi|=2_b}) is completed.
\end{proof}

\subsection{Statement of the main results}

In this subsection, we state the classification of all cases where $H$ is spherical and the SM-decomposition of~$\Psi$ is trivial.
Theorem~\ref{thm_par1psi1} reproduces~\cite[Theorem~5.15]{Avd_DSS} and lists all cases with~$|\Psi|=1$.
The main results of this paper are Theorems~\ref{thm_par1psi2} and~\ref{thm_par2psi2}, which provide a classification of all cases with $|\Psi| = 2$.

In view of the necessary conditions of Proposition~\ref{prop_|Psi|=1}, the next theorem classifies all cases where $H$ is spherical and~$|\Psi|=1$.
It also provides the corresponding sets of spherical roots in each case.

\begin{theorem} \label{thm_par1psi1}
Suppose that $G$ is simple of rank~$n$ and Dynkin type~$\mathsf X_n$, $\Pi = \lbrace \alpha_1, \ldots, \alpha_n \rbrace$, $\Pi_L = \Pi \setminus \lbrace \alpha_k \rbrace$ for some~$k \in \lbrace 1,\ldots,n\rbrace$, and $\Psi = \lbrace \overline \alpha_k \rbrace$.
Then the following assertions hold.
\begin{enumerate}[label=\textup{(\alph*)},ref=\textup{\alph*}]
\item \label{thm_par1psi1_a}
$H$ is a spherical subgroup of $G$ if and only if, up to an automorphism of the Dynkin diagram of~$G$, the pair $(\mathsf X_n, k)$ appears in Table~\textup{\ref{table_par1psi1}}.

\item \label{thm_par1psi1_b}
For each pair $(\mathsf X_n, k)$ listed in Table~\textup{\ref{table_par1psi1}} the set $\Sigma_G(G/H)$ is given in the last column of that table.
\end{enumerate}
\end{theorem}

\begin{table}[h]
\caption{Cases with $|\Pi \setminus \Pi_L| = 1$ and $|\Psi|=1$} \label{table_par1psi1}

\begin{center}
\renewcommand{\tabcolsep}{3pt}%
\begin{tabular}{|c|l|l|c|l|}
\hline
No. & \multicolumn{1}{|c|}{$(\mathsf X_n, k)$} & \multicolumn{1}{|c|}{($L', \mathfrak u$)} & $\rk$ & \multicolumn{1}{|c|}{$\Sigma_G(G/H)$} \\

\hline
\no\label{No1} &
\renewcommand{\tabcolsep}{0pt}%
\begin{tabular}{l}
($\mathsf A_{n},k$), \\
$n {\ge} 1$, $k {\le} (n{+}1)/2$
\end{tabular}
&
(${\SL_k} {\times} {\SL_{n+1-k}},
\KK^k {\otimes} \KK^{n+1-k}$) &
$k$ &
\renewcommand{\tabcolsep}{0pt}%
\begin{tabular}{l}
$\alpha_i {+} \alpha_{n+1-i}$ for $1 {\le} i {\le} k{-}1$, \\[-2pt]
$\alpha_k {+} \ldots {+} \alpha_{n+1-k}$
\end{tabular}
\\

\hline

\no\label{No2} &
($\mathsf B_{n},1$), $n {\ge} 3$ &
(${\SO_{2n-1}}, \KK^{2n-1}$) &
$2$ &
$\alpha_1$, $2\alpha_{2} {+} \ldots {+} 2\alpha_{n}$
\\

\hline

\no &
($\mathsf B_{n},n$), $n {\ge} 3$ &
(${\SL_{n}}, \KK^{n}$) &
$1$ &
$\alpha_{1} {+} \ldots {+} \alpha_{n}$
\\

\hline

\no\label{No4} &
($\mathsf C_{n}, 1$), $n {\ge} 2$ &
(${\Sp_{2n-2}}, \KK^{2n-2}$) &
$1$ &
$\alpha_1 {+} 2(\sum_{i=2}^{n-1} \alpha_i) {+} \alpha_n$
\\

\hline

\no\label{No5} &
($\mathsf C_{n}, 2$), $n {\ge} 4$ &
(${\SL_2}{\times}{\Sp_{2n-4}}, \KK^2{\otimes}\KK^{2n-4}$) &
$3$ &
\renewcommand{\tabcolsep}{0pt}%
\begin{tabular}{l}
$\alpha_1 {+} \alpha_3$, $\alpha_2$,\\
$\alpha_3 {+} 2(\sum_{i=4}^{n-1}\alpha_i) {+} \alpha_n$
\end{tabular}
\\

\hline

\no\label{No6} &
($\mathsf C_{5}, 3$) &
(${\SL_3}{\times}{\Sp_{4}}, \KK^{3}{\otimes}\KK^{4}$) &
$5$ &
$\alpha_1$, $\alpha_2$, $\alpha_3$, $\alpha_4$, $\alpha_5$
\\

\hline

\no\label{No7} &
($\mathsf C_{n}, 3$), $n {\ge} 6$ &
(${\SL_3}{\times}{\Sp_{2n-6}}, \KK^{3}{\otimes}\KK^{2n-6}$) &
$6$ &
\renewcommand{\tabcolsep}{0pt}%
\begin{tabular}{l}
$\alpha_1$, $\alpha_2$, $\alpha_3$, $\alpha_4$, $\alpha_5$,\\
$\alpha_5 {+} 2(\sum_{i=6}^{n-1}\alpha_i) {+} \alpha_n$
\end{tabular}
\\

\hline

\no\label{No8} &
($\mathsf C_{n}, n{-}2$), $n {\ge} 6$ &
(${\SL_{n-2}{\times}{\Sp_4}}, \KK^{n-2}{\otimes}\KK^4$) &
$6$ &
\renewcommand{\tabcolsep}{0pt}%
\begin{tabular}{l}
$\alpha_1$, $\alpha_2$, $\alpha_3$, $\alpha_{n-1}$, $\alpha_n$,\\[-2pt]
$\alpha_4 {+} \ldots {+} \alpha_{n-2}$
\end{tabular}
\\

\hline

\no\label{No9} &
($\mathsf C_{n}, {n{-}1}$), $n {\ge} 3$ &
(${\SL_{n-1}{\times}{\SL_2}}, \KK^{n-1}{\otimes}\KK^2$) &
$2$ &
\renewcommand{\tabcolsep}{0pt}%
\begin{tabular}{l}
$\alpha_{1}{+}\alpha_n$, $\alpha_{2} {+} \ldots {+} \alpha_{n-1}$
\end{tabular}
\\

\hline

\no\label{No10} &
($\mathsf C_{n}, {n}$), $n {\ge} 2$ &
(${\SL_{n}}, \mathrm{S}^2\KK^{n}$) &
$n$ &
\renewcommand{\tabcolsep}{0pt}%
\begin{tabular}{l}
$2\alpha_{i}$ for $1 {\le} i {\le} n{-}1$, $\alpha_{n}$
\end{tabular}
\\

\hline

\no\label{No11} &
($\mathsf D_{n}, 1$), $n {\ge} 4$ &
(${\SO_{2n-2}}, \KK^{2n-2}$) &
$2$ &
\renewcommand{\tabcolsep}{0pt}%
\begin{tabular}{l}
$\alpha_1$, \\[-2pt]
$2\alpha_{2} {+} \ldots {+} 2\alpha_{n-2} {+} \alpha_{n-1} {+} \alpha_{n}$
\end{tabular}
\\

\hline

\no\label{No12} &
($\mathsf D_n, n$), $n{=}2m{+}1 {\ge} 5$ &
(${\SL_{n}}, \wedge^2\KK^{n}$) &
$m$ &
\renewcommand{\tabcolsep}{0pt}%
\begin{tabular}{l}
$\alpha_{2i-1} {+} 2\alpha_{2i} {+} \alpha_{2i+1}$ \\[-2pt]
for $1 {\le} i {\le} m{-}1$, \\[-2pt]
$\alpha_{2m-1} {+} \alpha_{2m} {+} \alpha_{2m+1}$
\end{tabular}
\\

\hline

\no\label{No13} &
($\mathsf D_{n}, n$), $n{=}2m {\ge} 4$ &
(${\SL_{n}}, \wedge^2\KK^{n}$) &
$m$ &
\renewcommand{\tabcolsep}{0pt}%
\begin{tabular}{l}
$\alpha_{2i-1} {+} 2\alpha_{2i} {+} \alpha_{2i+1}$ \\[-2pt]
for $1 {\le} i {\le} m{-}1$, %\\
$\alpha_{2m}$
\end{tabular}
\\

\hline

\no &
$(\mathsf G_2, 1)$ &
$(\SL_2, \KK^2)$ &
$1$ &
$\alpha_1 {+} \alpha_2$
\\

\hline

\no\label{No15} &
$(\mathsf F_4, 3)$ &
$({\SL_3} {\times} {\SL_2}, \KK^3 {\otimes} \KK^2)$ &
$2$ &
$\alpha_1 {+} \alpha_4$, $\alpha_2 {+} \alpha_3$
\\

\hline

\no\label{No16} &
$(\mathsf F_4, 4)$ &
$(\Spin_7, R(\varpi_3))$ &
$2$ &
$\alpha_1 {+} 2\alpha_2 {+} 3\alpha_3$, $\alpha_4$
\\

\hline

\no\label{No17} &
$(\mathsf E_6, 6)$ &
$(\Spin_{10}, R(\varpi_5))$ &
$2$ &
\renewcommand{\tabcolsep}{0pt}%
\begin{tabular}{l}
$\alpha_1 {+} \alpha_3 {+} \alpha_4 {+} \alpha_5 {+} \alpha_6$,\\[-2pt]
$2\alpha_2 {+} \alpha_3 {+} 2\alpha_4 {+} \alpha_5$
\end{tabular}
\\

\hline

\no\label{No18} &
$(\mathsf E_7, 7)$ &
$(\mathsf E_6, R(\varpi_1))$ & 3 &
\renewcommand{\tabcolsep}{0pt}%
\begin{tabular}{l}
$2\alpha_1 {+} \alpha_2 {+} 2\alpha_3 {+} 2\alpha_4 {+} \alpha_5$,\\[-2pt]
$\alpha_2 {+} \alpha_3 {+} 2\alpha_4 {+} 2\alpha_5 {+} 2\alpha_6$, $\alpha_7$\\
\end{tabular}
\\

\hline
\end{tabular}
\end{center}
\end{table}

In Table~\ref{table_par1psi1}, $\varpi_i$ denotes the $i$th fundamental weight of~$G$ and $R(\lambda)$ stands for the simple $G$-module with highest weight~$\lambda$.

In view of the necessary conditions of Proposition~\ref{prop_|Psi|=2}, the next two theorems provide a classification of all cases where $H$ is spherical, $|\Psi|=2$, and the SM-decomposition of~$\Psi$ is trivial.
They also provide the corresponding sets of spherical roots for each case.

\begin{theorem} \label{thm_par1psi2}
Suppose that $G$ is simple of rank~$n$ and Dynkin type~$\mathsf X_n$, $\Pi_L = \Pi \setminus \lbrace \alpha_k \rbrace$ for some~$k \in \lbrace 1,\ldots,n\rbrace$, $\Psi = \lbrace p\overline \alpha_k, q\overline \alpha_k \rbrace$ for some $q > p \ge 1$, and the SM-decomposition of~$\Psi$ is trivial.
Then the following assertions hold.
\begin{enumerate}[label=\textup{(\alph*)},ref=\textup{\alph*}]
\item \label{thm_par1psi2_a}
$H$ is a spherical subgroup of $G$ if and only if $p = 1$ and, up to an automorphism of the Dynkin diagram of~$G$, the triple $(\mathsf X_n, k, q)$ appears in Table~\textup{\ref{table_par1psi2}}.

\item \label{thm_par1psi2_b}
For each triple $(\mathsf X_n, k, q)$ listed in Table~\textup{\ref{table_par1psi2}} the set $\Sigma_G(G/H)$ is given in the last column of that table.
\end{enumerate}
\end{theorem}

\begin{table}[h]
\caption{Cases with $|\Pi \setminus \Pi_L| = 1$ and $|\Psi|=2$}
\label{table_par1psi2}

\begin{center}
\renewcommand{\tabcolsep}{3pt}%
\begin{tabular}{|c|l|l|c|l|}
\hline
No. & \multicolumn{1}{|c|}{$(\mathsf X_n, k, q)$} & \multicolumn{1}{|c|}{($L', \mathfrak u$)} & $\rk$ & \multicolumn{1}{|c|}{$\Sigma_G(G/H)$} \\

\hline

\no\label{Bnn2} & $(\mathsf B_{n}, n, 2)$, $n{\ge}3$ & $(\SL_n, \KK^{n}{\oplus} \wedge^2 \KK^n)$ & $n$ &
\renewcommand{\tabcolsep}{0pt}%
\begin{tabular}{l}
$\alpha_i {+} \alpha_{i+1}$ for $1{\le}i{\le}n{-}1$, \\
$\alpha_n$ \\
\end{tabular}
\\

\hline

\no\label{F433} & $(\mathsf F_4, 3, 3)$ & $(\SL_{3} {\times} \SL_{2},\KK^{3} {\otimes} \KK^2 {\oplus} \KK^2)$ & $4$ & $\alpha_1, \alpha_2 +\alpha_3, \alpha_{3}, \alpha_4$  \\

\hline
\end{tabular}
\end{center}
\end{table}

\begin{theorem} \label{thm_par2psi2}
Suppose that $G$ is simple of rank~$n$, $\Pi_L = \Pi \setminus \lbrace \alpha_k, \alpha_l \rbrace$ for some~$k, l \in \lbrace 1,\ldots,n\rbrace$ with $k < l$, $\Psi = \lbrace p \overline \alpha_k + q \overline \alpha_l, r \overline \alpha_k + s \overline \alpha_l \rbrace$ for some $p,q,r,s \in \ZZ_{\ge0}$, and the SM-decomposition of~$\Psi$ is trivial.
Then the following assertions hold.
\begin{enumerate}[label=\textup{(\alph*)},ref=\textup{\alph*}]
\item \label{thm_par2psi2_a}
$H$ is a spherical subgroup of $G$ if and only if, up to an automorphism of the Dynkin diagram of~$G$ and up to interchanging the pairs $(p,q)$ and $(r,s)$, the collection of pairs $(k,l), (p,q),(r,s)$ appears in one of the following tables depending on the type of~$G$:
\begin{itemize}
\item Table~\textup{\ref{table_A}} if $G$ is of type~$\mathsf A_n$ \textup($n \ge 3$\textup);
\item Table~\textup{\ref{table_B}} if $G$ is of type~$\mathsf B_n$ \textup($n \ge 3$\textup);
\item Table~\textup{\ref{table_C}} if $G$ is of type~$\mathsf C_n$ \textup($n \ge 3$\textup);
\item Table~\textup{\ref{table_D}} if $G$ is of type~$\mathsf D_n$ \textup($n \ge 4$\textup);
\item Table~\textup{\ref{table_F4}} if $G$ is of type~$\mathsf F_4$;
\item Table~\textup{\ref{table_E6}} if $G$ is of type~$\mathsf E_6$;
\item Table~\textup{\ref{table_E7}} if $G$ is of type~$\mathsf E_7$;
\item Table~\textup{\ref{table_E8}} if $G$ is of type~$\mathsf E_8$.
\end{itemize}

\item \label{thm_par2psi2_b}
For each of the cases in part~\textup(\ref{thm_par2psi2_a}\textup), the set $\Sigma_G(G/H)$ is given in the last column of the corresponding table.
\end{enumerate}
\end{theorem}

\begin{remark}
For convenience of the reader, for each  Table~\ref{table_par1psi1}--\ref{table_E8} we also included the information on the $L'$-module structure of~$\mathfrak u$ (up to equivalence) as well as the value of rank of~$\mathfrak u$ as a spherical $L$-module, which also equals the cardinality of the set~$\Sigma_G(G/H)$ by Proposition~\ref{prop_num_sr}.
\end{remark}

\begin{remark}
It can happen for Tables~\ref{table_A}--\ref{table_E8} that for two or more cases belonging to the same table the corresponding spherical subgroups are conjugate in~$G$.
Since this is not important for our algorithms, we do not make an attempt to identify all such cases.
\end{remark}

\begin{table}[h]
\caption{Cases with $|\Pi \setminus \Pi_L| = 2$ and $|\Psi|=2$ for type~$\mathsf A_n$ ($n \ge 3$)} \label{table_A}

\begin{center}
\renewcommand{\tabcolsep}{3pt}%
\begin{tabular}{|c|l|l|l|c|l|}
\hline
No. & \multicolumn{1}{|c|}{$(k,l)$} & \multicolumn{1}{|c|}{$(p,q),(r,s)$} & \multicolumn{1}{|c|}{($L', \mathfrak u$)} & $\rk$ & \multicolumn{1}{|c|}{$\Sigma_G(G/H)$} \\

\hline

\no\label{An1} &
\renewcommand{\tabcolsep}{0pt}%
\begin{tabular}{l}
$(k,n)$, \\
$n{-}k {>} k {\ge} 1$
\end{tabular}
& $(1,0),(0,1)$ &
\renewcommand{\tabcolsep}{0pt}%
\begin{tabular}{l}
$(\SL_k {\times} \SL_{n-k},$ \\
$\KK^{k}{\otimes}\KK^{n-k}{\oplus}(\KK^{n-k})^*)$
\end{tabular}
& $2k{+}1$ &
\renewcommand{\tabcolsep}{0pt}%
\begin{tabular}{l}
$\alpha_1,\ldots,\alpha_k$, \\
$\alpha_{k+1}{+}\ldots{+}\alpha_{n-k}$, \\
$\alpha_{n-k+1},\ldots,\alpha_{n}$
\end{tabular}
\\

\hline

\no\label{An2} &
\renewcommand{\tabcolsep}{0pt}%
\begin{tabular}{l}
$(k,n)$, \\
$k {\ge} n{-}k {\ge} 2$
\end{tabular}
& $(1,0),(0,1)$ &
\renewcommand{\tabcolsep}{0pt}%
\begin{tabular}{l}
$(\SL_k {\times} \SL_{n-k},$ \\
$\KK^{k}{\otimes}\KK^{n-k}{\oplus}(\KK^{n-k})^*)$
\end{tabular}
& $2(n{-}k)$ &
\renewcommand{\tabcolsep}{0pt}%
\begin{tabular}{l}
$\alpha_1,\ldots,\alpha_{n-k-1}$, \\
$\alpha_{n-k}{+}\ldots{+}\alpha_{k}$, \\
$\alpha_{k+1},\ldots,\alpha_n$
\end{tabular}
\\

\hline

\no\label{An3} &
\renewcommand{\tabcolsep}{0pt}%
\begin{tabular}{l}
$(k,n)$, \\
$n{-}k {\ge} k {\ge} 2$
\end{tabular}
& $(1,0),(1,1)$ &
\renewcommand{\tabcolsep}{0pt}%
\begin{tabular}{l}
$(\SL_{k}{\otimes}\SL_{n-k},$ \\
$\KK^{k}{\otimes}\KK^{n-k} {\oplus} \KK^{k})$
\end{tabular}
& $2k$  &
\renewcommand{\tabcolsep}{0pt}%
\begin{tabular}{l}
$\alpha_1,\ldots,\alpha_{k-1}$, \\
$\alpha_{k}{+}\ldots{+}\alpha_{n-k}$, \\
$\alpha_{n-k+1},\ldots,\alpha_{n}$
\end{tabular}
\\

\hline

\no\label{An4} &
\renewcommand{\tabcolsep}{0pt}%
\begin{tabular}{l}
$(k,n)$, \\
$k {>} n{-}k {\ge} 1$
\end{tabular}
& $(1,0),(1,1)$ &
\renewcommand{\tabcolsep}{0pt}%
\begin{tabular}{l}
$(\SL_{k}{\otimes}\SL_{n-k},$ \\
$\KK^{k}{\otimes}\KK^{n-k} {\oplus} \KK^{k})$
\end{tabular}
& $2(n{-}k){+}1$  &
\renewcommand{\tabcolsep}{0pt}%
\begin{tabular}{l}
$\alpha_1,\ldots,\alpha_{n-k}$,\\
$\alpha_{n-k+1}{+}\ldots{+}\alpha_{k}$, \\
$\alpha_{k+1},\ldots,\alpha_n$
\end{tabular}
\\

\hline

\no\label{An5} &
\renewcommand{\tabcolsep}{0pt}%
\begin{tabular}{l}
$(k,k{+}1)$, \\
$n{-}k {\ge} k {\ge}2$
\end{tabular}
& $(1,0),(1,1)$ &
\renewcommand{\tabcolsep}{0pt}%
\begin{tabular}{l}
$(\SL_{k}{\otimes}\SL_{n-k},$ \\
$\KK^{k} {\oplus} \KK^{k}{\otimes}\KK^{n-k})$
\end{tabular}
& $2k$  &
\renewcommand{\tabcolsep}{0pt}%
\begin{tabular}{l}
$\alpha_1,\ldots,\alpha_k$, \\
$\alpha_{k+1}{+}\ldots{+}\alpha_{n-k+1}$, \\
$\alpha_{n-k+2},\ldots,\alpha_n$
\end{tabular}
\\

\hline

\no\label{An6} &
\renewcommand{\tabcolsep}{0pt}%
\begin{tabular}{l}
$(k,k{+}1)$, \\
$k {>} n{-}k {\ge}2$
\end{tabular}
& $(1,0),(1,1)$ &
\renewcommand{\tabcolsep}{0pt}%
\begin{tabular}{l}
$(\SL_{k}{\otimes}\SL_{n-k},$ \\
$\KK^{k} {\oplus} \KK^{k}{\otimes}\KK^{n-k})$
\end{tabular}
& $2(n{-}k){+}1$  &
\renewcommand{\tabcolsep}{0pt}%
\begin{tabular}{l}
$\alpha_1,\ldots,\alpha_{n-k}$,\\
$\alpha_{n-k+1}{+}\ldots{+}\alpha_{k}$, \\
$\alpha_{k+1},\ldots,\alpha_n$
\end{tabular}
\\

\hline

\no\label{An7} &
\renewcommand{\tabcolsep}{0pt}%
\begin{tabular}{l}
$(k,k{+}2)$, \\
$2 {\le} k {\le} n{-}3$
\end{tabular}
& $(1,0),(0,1)$ &
\renewcommand{\tabcolsep}{0pt}%
\begin{tabular}{l}
$(\SL_k \times \SL_2 \times \SL_{n-k-1},$ \\
$\KK^k{\otimes}\KK^2 {\oplus} \KK^2{\otimes}\KK^{n-k-1})$
\end{tabular}
& $5$ &
\renewcommand{\tabcolsep}{0pt}%
\begin{tabular}{l}
$\alpha_1,\alpha_2 {+} \ldots {+} \alpha_{k}$, \\
$\alpha_{k+1},$ \\
$\alpha_{k+2} {+} \ldots {+} \alpha_{n-1}, \alpha_n$
\end{tabular}
\\

\hline

\no\label{An8} &
\renewcommand{\tabcolsep}{0pt}%
\begin{tabular}{l}
$(2,l)$, \\
$4 {\le} l {\le} n{-}1$
\end{tabular}
& $(1,0),(1,1)$ &
\renewcommand{\tabcolsep}{0pt}%
\begin{tabular}{l}
$(\SL_{2} {\times} \SL_{l-2} {\times} \SL_{n-l+1}$, \\
$\KK^2 {\otimes} \KK^{l-2} {\oplus} \KK^2{\otimes}\KK^{n-l+1})$
\end{tabular}
& $5$ &
\renewcommand{\tabcolsep}{0pt}%
\begin{tabular}{l}
$\alpha_1, \alpha_2 {+} \ldots {+} \alpha_{l-2}$, \\
$\alpha_{l-1}$, \\
$\alpha_{l} {+} \ldots {+} \alpha_{n-1}, \alpha_n$
\end{tabular}
\\

\hline
\end{tabular}
\end{center}
\end{table}

\begin{table}[h]
\caption{Cases with $|\Pi \setminus \Pi_L| = 2$ and $|\Psi|=2$ for type~$\mathsf B_n$ ($n \ge 3$)} \label{table_B}

\begin{center}
\renewcommand{\tabcolsep}{3pt}%
\begin{tabular}{|c|l|l|l|c|l|}
\hline
No. & \multicolumn{1}{|c|}{$(k,l)$} & \multicolumn{1}{|c|}{$(p,q),(r,s)$} & \multicolumn{1}{|c|}{($L', \mathfrak u$)} & $\rk$ & \multicolumn{1}{|c|}{$\Sigma_G(G/H)$} \\

\hline

\no\label{Bn1} &
\renewcommand{\tabcolsep}{0pt}%
\begin{tabular}{l}
$(k,n)$, \\
$n{-}k {>} k {\ge} 1$
\end{tabular}
& $(1,0),(0,1)$ &
\renewcommand{\tabcolsep}{0pt}%
\begin{tabular}{l}
$(\SL_k {\times} \SL_{n-k},$ \\
$\KK^{k}{\otimes}\KK^{n-k}{\oplus}(\KK^{n-k})^*)$
\end{tabular}
& $2k{+}1$ &
\renewcommand{\tabcolsep}{0pt}%
\begin{tabular}{l}
$\alpha_1,\ldots,\alpha_k$, \\
$\alpha_{k+1}{+}\ldots{+}\alpha_{n-k}$, \\
$\alpha_{n-k+1},\ldots,\alpha_{n}$
\end{tabular}
\\

\hline

\no\label{Bn2} &
\renewcommand{\tabcolsep}{0pt}%
\begin{tabular}{l}
$(k,n)$, \\
$k {\ge} n{-}k {\ge} 2$
\end{tabular}
& $(1,0),(0,1)$ &
\renewcommand{\tabcolsep}{0pt}%
\begin{tabular}{l}
$(\SL_k {\times} \SL_{n-k},$ \\
$\KK^{k}{\otimes}\KK^{n-k}{\oplus}(\KK^{n-k})^*)$
\end{tabular}
& $2(n{-}k)$ &
\renewcommand{\tabcolsep}{0pt}%
\begin{tabular}{l}
$\alpha_1,\ldots,\alpha_{n-k-1}$, \\
$\alpha_{n-k}{+}\ldots{+}\alpha_{k}$, \\
$\alpha_{k+1},\ldots,\alpha_n$
\end{tabular}
\\

\hline

\no\label{Bn3} &
\renewcommand{\tabcolsep}{0pt}%
\begin{tabular}{l}
$(k,n)$, \\
$n{-}k {\ge} k {\ge} 2$
\end{tabular}
& $(1,0),(1,1)$ &
\renewcommand{\tabcolsep}{0pt}%
\begin{tabular}{l}
$(\SL_{k}{\otimes}\SL_{n-k}$, \\
$\KK^{k}{\otimes}\KK^{n-k} {\oplus} \KK^{k})$
\end{tabular}
& $2k$  &
\renewcommand{\tabcolsep}{0pt}%
\begin{tabular}{l}
$\alpha_1,\ldots,\alpha_{k-1}$, \\
$\alpha_{k}{+}\ldots{+}\alpha_{n-k}$, \\
$\alpha_{n-k+1},\ldots,\alpha_{n}$
\end{tabular}
\\

\hline

\no\label{Bn4} &
\renewcommand{\tabcolsep}{0pt}%
\begin{tabular}{l}
$(k,n)$, \\
$k {>} n{-}k {\ge} 1$
\end{tabular}
& $(1,0),(1,1)$ &
\renewcommand{\tabcolsep}{0pt}%
\begin{tabular}{l}
$(\SL_{k}{\otimes}\SL_{n-k}$, \\
$\KK^{k}{\otimes}\KK^{n-k} {\oplus} \KK^{k})$
\end{tabular}
& $2(n{-}k){+}1$  &
\renewcommand{\tabcolsep}{0pt}%
\begin{tabular}{l}
$\alpha_1,\ldots,\alpha_{n-k}$,\\
$\alpha_{n-k+1}{+}\ldots{+}\alpha_{k}$, \\
$\alpha_{k+1},\ldots,\alpha_n$
\end{tabular}
\\

\hline
\end{tabular}
\end{center}
\end{table}

\begin{table}%[h]
\caption{Cases with $|\Pi \setminus \Pi_L| = 2$ and $|\Psi|=2$ for type~$\mathsf C_n$ ($n \ge 3$)} \label{table_C}

\begin{center}
\renewcommand{\tabcolsep}{3pt}%
\begin{tabular}{|c|l|l|l|c|l|}
\hline
No. & \multicolumn{1}{|c|}{$(k,l)$} & \multicolumn{1}{|c|}{$(p,q),(r,s)$} & \multicolumn{1}{|c|}{($L', \mathfrak u$)} & $\rk$ & \multicolumn{1}{|c|}{$\Sigma_G(G/H)$} \\

\hline

\no\label{Cn1} &
\renewcommand{\tabcolsep}{0pt}%
\begin{tabular}{l}
$(1,2)$, \\
$n{=}3$
\end{tabular}
& $(0,1),(1,1)$ & $(\SL_{2},\KK^{2} {\oplus} \KK^{2})$ & $3$ & $\alpha_1, \alpha_2, \alpha_3$
\\

\hline

\no\label{Cn2} &
\renewcommand{\tabcolsep}{0pt}%
\begin{tabular}{l}
$(1,2)$, \\
$n{\ge}4$
\end{tabular}
& $(0,1),(1,1)$ & $(\Sp_{2n-4},\KK^{2n-4} {\oplus} \KK^{2n-4})$ & $4$ &
\renewcommand{\tabcolsep}{0pt}%
\begin{tabular}{l}
$\alpha_1,\alpha_2, \alpha_3,$ \\
$\alpha_3 {+} 2(\sum_{i=4}^{n-1} \alpha_i) {+} \alpha_n$
\end{tabular}
\\

\hline

\no\label{Cn3} &
\renewcommand{\tabcolsep}{0pt}%
\begin{tabular}{l}
$(1,3)$, \\
$n {=} 4$
\end{tabular}
& $(1,0),(0,1)$ & $(\SL_2 {\times} \SL_{2},\KK^{2} {\oplus} \KK^2 {\otimes} \KK^{2})$ & $4$ &
$\alpha_1,\alpha_2,\alpha_3,\alpha_4$
\\

\hline

\no\label{Cn4} &
\renewcommand{\tabcolsep}{0pt}%
\begin{tabular}{l}
$(1,3)$, \\
$n {\ge} 5$
\end{tabular}
& $(1,0),(0,1)$ &
\renewcommand{\tabcolsep}{0pt}%
\begin{tabular}{l}
$(\SL_2 {\times} \Sp_{2n-6},$ \\
$\KK^{2} {\oplus} \KK^2 {\otimes} \KK^{2n-6})$
\end{tabular}
& $5$ &
\renewcommand{\tabcolsep}{0pt}%
\begin{tabular}{l}
$\alpha_1, \alpha_2, \alpha_3, \alpha_4,$ \\
$\alpha_4 {+} 2(\sum_{i=5}^{n-1} \alpha_i) {+} \alpha_n$
\end{tabular}
\\

\hline

\no\label{Cn5} &
\renewcommand{\tabcolsep}{0pt}%
\begin{tabular}{l}
$(1,n{-}1)$, \\
$n {\ge} 5$
\end{tabular}
& $(1,0),(0,1)$ &
\renewcommand{\tabcolsep}{0pt}%
\begin{tabular}{l}
$(\SL_{n-2} {\times} \SL_{2}$, \\
$(\KK^{n-2})^* {\oplus} \KK^{n-2} {\otimes} \KK^{2})$
\end{tabular}
& $5$ &
\renewcommand{\tabcolsep}{0pt}%
\begin{tabular}{l}
$\alpha_1, \alpha_2, \alpha_3{+}\ldots{+}\alpha_{n-2},$ \\
$\alpha_{n-1},\alpha_n$
\end{tabular}
\\

\hline

\no\label{Cn6} &
\renewcommand{\tabcolsep}{0pt}%
\begin{tabular}{l}
$(1,n{-}1)$, \\
$n{\ge}4$
\end{tabular}
& $(0,1),(1,1)$ &
\renewcommand{\tabcolsep}{0pt}%
\begin{tabular}{l}
$(\SL_{n-2} {\times} \SL_2$, \\
$\KK^{n-2}{\otimes}\KK^2 {\oplus} \KK^{2})$
\end{tabular}
& $4$ &
\renewcommand{\tabcolsep}{0pt}%
\begin{tabular}{l}
$\alpha_1, \alpha_2,$ \\
$\alpha_{3}{+}\ldots{+}\alpha_{n-1}, \alpha_n$
\end{tabular}
\\

\hline

\no\label{Cn7} &
$(1,n)$ & $(1,0),(1,1)$ & $(\SL_{n-1}, \KK^{n-1} \oplus (\KK^{n-1})^*)$ & $3$ & $\alpha_1, \alpha_2{+}\ldots{+}\alpha_{n-1}, \alpha_n$ \\

\hline

\no\label{Cn8} &
\renewcommand{\tabcolsep}{0pt}%
\begin{tabular}{l}
$(2,3)$, \\
$n {=} 4$
\end{tabular}
& $(1,0),(1,1)$ & $(\SL_2 {\times} \SL_{2},\KK^{2} {\oplus} \KK^2 {\otimes} \KK^{2})$ & $4$ &
$\alpha_1,\alpha_2,\alpha_3,\alpha_4$
\\

\hline

\no\label{Cn9} &
\renewcommand{\tabcolsep}{0pt}%
\begin{tabular}{l}
$(2,3)$, \\
$n {\ge} 5$
\end{tabular}
& $(1,0),(1,1)$ &
\renewcommand{\tabcolsep}{0pt}%
\begin{tabular}{l}
$(\SL_2 {\times} \Sp_{2n-6},$ \\
$\KK^{2} {\oplus} \KK^2 {\otimes} \KK^{2n-6})$
\end{tabular}
& $5$ &
\renewcommand{\tabcolsep}{0pt}%
\begin{tabular}{l}
$\alpha_1, \alpha_2, \alpha_3, \alpha_4,$ \\
$\alpha_4 {+} 2(\sum_{i=5}^{n-1} \alpha_i) {+} \alpha_n$
\end{tabular}
\\

\hline

\no\label{Cn10} &
\renewcommand{\tabcolsep}{0pt}%
\begin{tabular}{l}
$(2,l)$, \\
$4{\le}l{\le}n{-}2$
\end{tabular}
& $(1,0),(1,1)$ &
\renewcommand{\tabcolsep}{0pt}%
\begin{tabular}{l}
$(\SL_{2} {\times} \SL_{l-2} {\times} \Sp_{2n-2l}$, \\
$\KK^{2}{\otimes}\KK^{l-2} {\oplus} \KK^2 {\otimes} \KK^{2n-2l})$
\end{tabular}
& $6$ &
\renewcommand{\tabcolsep}{0pt}%
\begin{tabular}{l}
$\alpha_1,\alpha_2{+}\ldots{+}\alpha_{l-2}$, \\
$\alpha_{l-1}, \alpha_{l}, \alpha_{l+1}$, \\
$\alpha_{l+1} {+} 2(\sum_{i=l+2}^{n-1} \alpha_i) {+} \alpha_n$ \\
\end{tabular}
\\

\hline

\no\label{Cn11} &
\renewcommand{\tabcolsep}{0pt}%
\begin{tabular}{l}
$(2,n{-}1)$, \\
$n{\ge}5$
\end{tabular}
& $(1,0),(1,1)$ &
\renewcommand{\tabcolsep}{0pt}%
\begin{tabular}{l}
$(\SL_{2}{\times}\SL_{n-3} {\times} \SL_{2}$, \\
$\KK^{2}{\otimes}\KK^{n-3} {\oplus} \KK^2 {\otimes} \KK^{2})$
\end{tabular}
& $5$ &
\renewcommand{\tabcolsep}{0pt}%
\begin{tabular}{l}
$\alpha_1, \alpha_2{+}\ldots{+}\alpha_{n-3}$, \\
$\alpha_{n-2}, \alpha_{n-1}, \alpha_{n}$
\end{tabular}
\\

\hline

\no\label{Cn12} &
\renewcommand{\tabcolsep}{0pt}%
\begin{tabular}{l}
$(k,k{+}2)$, \\
$2{\le}k{\le}n{-}4$
\end{tabular}
& $(1,0),(0,1)$ &
\renewcommand{\tabcolsep}{0pt}%
\begin{tabular}{l}
$(\SL_k{\times}\SL_2 {\times} \Sp_{2n-2k-4}$, \\
$\KK^{k}{\otimes}\KK^{2} {\oplus} \KK^2 {\otimes} \KK^{2n-2k-4})$
\end{tabular}
& $6$ &
\renewcommand{\tabcolsep}{0pt}%
\begin{tabular}{l}
$\alpha_1,\alpha_2{+}\ldots{+}\alpha_{k}$, \\
$\alpha_{k+1}, \alpha_{k+2}, \alpha_{k+3}$, \\
$\alpha_{k+3} {+} 2(\sum_{i=k+4}^{n-1} \alpha_i) {+} \alpha_n$ \\
\end{tabular}
\\

\hline

\no\label{Cn13} &
\renewcommand{\tabcolsep}{0pt}%
\begin{tabular}{l}
$(k,n{-}1)$, \\
$2 {\le} k {\le} n{-}3$
\end{tabular}
& $(0,1),(1,1)$ &
\renewcommand{\tabcolsep}{0pt}%
\begin{tabular}{l}
$(\SL_{k} {\times} \SL_{n-k-1} {\times} \SL_2,$ \\
$\KK^{n-k-1} {\otimes} \KK^{2} {\oplus} \KK^{k}{\otimes}\KK^2)$
\end{tabular}
& $5$ &
\renewcommand{\tabcolsep}{0pt}%
\begin{tabular}{l}
$\alpha_1, \alpha_2{+}\ldots{+}\alpha_{k}, \alpha_{k+1},$ \\
$\alpha_{k+2}{+}\ldots{+}\alpha_{n-1}, \alpha_n$
\end{tabular}
\\

\hline

\no\label{Cn14} &
\renewcommand{\tabcolsep}{0pt}%
\begin{tabular}{l}
$(n{-}3,n{-}1)$, \\
$n{\ge}5$
\end{tabular}
& $(1,0),(0,1)$ &
\renewcommand{\tabcolsep}{0pt}%
\begin{tabular}{l}
$(\SL_{n-3}{\times}\SL_2 {\times} \SL_{2}$, \\
$\KK^{n-3}{\otimes}\KK^{2} {\oplus} \KK^2 {\otimes} \KK^{2})$
\end{tabular}
& $5$ &
\renewcommand{\tabcolsep}{0pt}%
\begin{tabular}{l}
$\alpha_1, \alpha_2{+}\ldots{+}\alpha_{n-3}$, \\
$\alpha_{n-2}, \alpha_{n-1}, \alpha_{n}$
\end{tabular}
\\

\hline

\no\label{Cn15} &
\renewcommand{\tabcolsep}{0pt}%
\begin{tabular}{l}
$(n{-}2,n{-}1)$, \\
$n {\ge} 5$
\end{tabular}
& $(1,0),(1,1)$ &
\renewcommand{\tabcolsep}{0pt}%
\begin{tabular}{l}
$(\SL_{n-2} {\times} \SL_{2}$, \\
$\KK^{n-2} {\oplus} \KK^{n-2} {\otimes} \KK^{2})$
\end{tabular}
& $5$ &
\renewcommand{\tabcolsep}{0pt}%
\begin{tabular}{l}
$\alpha_1, \alpha_2, \alpha_3{+}\ldots{+}\alpha_{n-2},$ \\
$\alpha_{n-1},\alpha_n$
\end{tabular}
\\

\hline

\no\label{Cn16} &
\renewcommand{\tabcolsep}{0pt}%
\begin{tabular}{l}
$(n{-}2,n{-}1)$, \\
$n{\ge}4$
\end{tabular}
& $(0,1),(1,1)$ &
\renewcommand{\tabcolsep}{0pt}%
\begin{tabular}{l}
$(\SL_{n-2} {\times} \SL_2$, \\
$\KK^{2} {\oplus} \KK^{n-2}{\otimes}\KK^2)$
\end{tabular}
& $4$ &
\renewcommand{\tabcolsep}{0pt}%
\begin{tabular}{l}
$\alpha_1, \alpha_2{+}\ldots{+}\alpha_{n-2}$, \\
$\alpha_{n-1},\alpha_n$
\end{tabular}
\\

\hline

\no\label{Cn17} &
$(n{-}1,n)$ & $(1,0),(1,1)$ & $(\SL_{n-1}, \KK^{n-1} \oplus \KK^{n-1})$ & $3$ & $\alpha_1, \alpha_2{+}\ldots{+}\alpha_{n-1}, \alpha_n$ \\

\hline
\end{tabular}
\end{center}
\end{table}

\begin{table}[h]
\caption{Cases with $|\Pi \setminus \Pi_L| = 2$ and $|\Psi|=2$ for type~$\mathsf D_n$ ($n \ge 3$)} \label{table_D}

\begin{center}
\renewcommand{\tabcolsep}{3pt}%
\begin{tabular}{|c|l|l|l|c|l|}
\hline
No. & \multicolumn{1}{|c|}{$(k,l)$} & \multicolumn{1}{|c|}{$(p,q),(r,s)$} & \multicolumn{1}{|c|}{($L', \mathfrak u$)} & $\rk$ & \multicolumn{1}{|c|}{$\Sigma_G(G/H)$} \\

\hline

\no\label{Dn1} &
\renewcommand{\tabcolsep}{0pt}%
\begin{tabular}{l}
$(1,n)$ \\
$n{=}2m{+}1$
\end{tabular}
& $(1,0),(0,1)$ & $(\SL_{n-1},(\KK^{n-1})^* {\oplus} \wedge^2 \KK^{n-1})$ & $n{-}1$ &
\renewcommand{\tabcolsep}{0pt}%
\renewcommand{\tabcolsep}{0pt}%
\begin{tabular}{l}
$\alpha_i{+}\alpha_{i+1}$ for $1{\le}i{\le}n{-}2$, \\
$\alpha_n$
\end{tabular}
\\

\hline

\no\label{Dn2} &
\renewcommand{\tabcolsep}{0pt}%
\begin{tabular}{l}
$(1,n)$ \\
$n{=}2m$
\end{tabular}
& $(1,0),(0,1)$ & $(\SL_{n-1},(\KK^{n-1})^* {\oplus} \wedge^2 \KK^{n-1})$ & $n{-}1$ &
\renewcommand{\tabcolsep}{0pt}%
\renewcommand{\tabcolsep}{0pt}%
\begin{tabular}{l}
$\alpha_i{+}\alpha_{i+1}$ for $1{\le}i{\le}n{-}3$, \\
$\alpha_{n-2}{+}\alpha_{n}, \alpha_{n-1}$
\end{tabular}
\\

\hline

\no\label{Dn3} &
$(1,n)$
& $(1,0),(1,1)$ & $(\SL_{n-1},\KK^{n-1} {\oplus} (\KK^{n-1})^*)$ & $3$ &
\renewcommand{\tabcolsep}{0pt}%
\begin{tabular}{l}
$\alpha_1, \alpha_2{+}\ldots{+}\alpha_{n-1},$ \\
$\alpha_2{+}\ldots{+}\alpha_{n-2}{+}\alpha_n$
\end{tabular}
\\

\hline

\no\label{Dn4} &
\renewcommand{\tabcolsep}{0pt}%
\begin{tabular}{l}
$(1,n)$ \\
\end{tabular}
& $(0,1),(1,1)$ & $(\SL_{n-1}, \wedge^2 \KK^{n-1} {\oplus} \KK^{n-1})$ & $n{-}1$ &
\renewcommand{\tabcolsep}{0pt}%
\renewcommand{\tabcolsep}{0pt}%
\begin{tabular}{l}
$\alpha_i{+}\alpha_{i+1}$ for $1{\le}i{\le}n{-}2$, \\
$\alpha_n$
\end{tabular}
\\

\hline

\no\label{Dn5} &
\renewcommand{\tabcolsep}{0pt}%
\begin{tabular}{l}
$(2,5)$, \\
$n{=}5$
\end{tabular}
& $(1,0),(0,1)$ & $(\SL_{2}{\times}{\SL_3},\KK^{2}{\otimes}\KK^3 {\oplus} \KK^{3})$ & $5$ &
$\alpha_1, \alpha_2, \alpha_3, \alpha_4, \alpha_5$
\\

\hline

\no\label{Dn6} &
\renewcommand{\tabcolsep}{0pt}%
\begin{tabular}{l}
$(2,5)$, \\
$n{=}5$
\end{tabular}
& $(0,1),(1,1)$ & $(\SL_{2}{\times}{\SL_3}, (\KK^{3})^* {\oplus} \KK^{2}{\otimes}\KK^3)$ & $5$ &
$\alpha_1, \alpha_2, \alpha_3, \alpha_4, \alpha_5$
\\

\hline

\no\label{Dn7} &
\renewcommand{\tabcolsep}{0pt}%
\begin{tabular}{l}
$(3,5)$, \\
$n{=}5$
\end{tabular}
& $(1,0),(2,1)$ & $(\SL_{3}{\times}{\SL_2},\KK^{3}{\otimes}\KK^2 {\oplus} (\KK^{3})^*)$ & $5$ &
$\alpha_1, \alpha_2, \alpha_3, \alpha_4, \alpha_5$
\\

\hline

\no\label{Dn8} &
\renewcommand{\tabcolsep}{0pt}%
\begin{tabular}{l}
$(3,n)$, \\
$n{\ge}6$
\end{tabular}
& $(1,0),(2,1)$ &
\renewcommand{\tabcolsep}{0pt}%
\begin{tabular}{l}
$(\SL_{3}{\times}{\SL_{n-3}},$ \\
$\KK^{3}{\otimes}\KK^{n-3} {\oplus} (\KK^{3})^*)$
\end{tabular}
& $6$ &
\renewcommand{\tabcolsep}{0pt}%
\begin{tabular}{l}
$\alpha_1, \alpha_2, \alpha_3{+}\ldots{+}\alpha_{n-3},$ \\
$\alpha_{n-2}, \alpha_{n-1}, \alpha_n$
\end{tabular}
\\

\hline

\no\label{Dn9} &
\renewcommand{\tabcolsep}{0pt}%
\begin{tabular}{l}
$(n{-}3,n)$, \\
$n{\ge}6$
\end{tabular}
& $(1,0),(0,1)$ & $(\SL_{n-3}{\times}{\SL_3},\KK^{n-3}{\otimes}\KK^3 {\oplus} \KK^{3})$ & $6$ &
\renewcommand{\tabcolsep}{0pt}%
\begin{tabular}{l}
$\alpha_1, \alpha_2, \alpha_3{+}\ldots{+}\alpha_{n-3},$ \\
$\alpha_{n-2}, \alpha_{n-1}, \alpha_n$
\end{tabular}
\\

\hline

\no\label{Dn10} &
\renewcommand{\tabcolsep}{0pt}%
\begin{tabular}{l}
$(n{-}3,n)$, \\
$n{\ge}6$
\end{tabular}
& $(0,1),(1,1)$ &
\renewcommand{\tabcolsep}{0pt}%
\begin{tabular}{l}
$(\SL_{n-3}{\times}{\SL_3},$ \\
$(\KK^{3})^* {\oplus} \KK^{n-3}{\otimes}\KK^3)$
\end{tabular}
& $6$ &
\renewcommand{\tabcolsep}{0pt}%
\begin{tabular}{l}
$\alpha_1, \alpha_2, \alpha_3{+}\ldots{+}\alpha_{n-3},$ \\
$\alpha_{n-2}, \alpha_{n-1}, \alpha_n$
\end{tabular}
\\

\hline

\no\label{Dn11} &
$(n{-}1,n)$
& $(1,0),(0,1)$ & $(\SL_{n-1},\KK^{n-1} {\oplus} \KK^{n-1})$ & $3$ &
\renewcommand{\tabcolsep}{0pt}%
\begin{tabular}{l}
$\alpha_1, \alpha_2{+}\ldots{+}\alpha_{n-1},$ \\
$\alpha_2{+}\ldots{+}\alpha_{n-2}{+}\alpha_n$
\end{tabular}
\\

\hline

\no\label{Dn12} &
\renewcommand{\tabcolsep}{0pt}%
\begin{tabular}{l}
$(n{-}1,n)$, \\
$n{=}2m{+}1$
\end{tabular}
& $(1,0),(1,1)$ & $(\SL_{n-1},\KK^{n-1} {\oplus} \wedge^2 \KK^{n-1})$ & $n{-}1$ &
\renewcommand{\tabcolsep}{0pt}%
\begin{tabular}{l}
$\alpha_i{+}\alpha_{i+1}$ for $1{\le}i{\le}n{-}2$, \\
$\alpha_n$
\end{tabular}
\\

\hline

\no\label{Dn13} &
\renewcommand{\tabcolsep}{0pt}%
\begin{tabular}{l}
$(n{-}1,n)$, \\
$n{=}2m$
\end{tabular}
& $(1,0),(1,1)$ & $(\SL_{n-1},\KK^{n-1} {\oplus} \wedge^2 \KK^{n-1})$ & $n{-}1$ &
\renewcommand{\tabcolsep}{0pt}%
\begin{tabular}{l}
$\alpha_i{+}\alpha_{i+1}$ for $1{\le}i{\le}n{-}3$, \\
$\alpha_{n-2}{+}\alpha_n, \alpha_{n-1}$
\end{tabular}
\\

\hline
\end{tabular}
\end{center}
\end{table}

\begin{table}[h]
\caption{Cases with $|\Pi \setminus \Pi_L| = 2$ and $|\Psi|=2$ for type~$\mathsf F_4$} \label{table_F4}

\begin{center}
\renewcommand{\tabcolsep}{3pt}%
\begin{tabular}{|c|l|l|l|c|l|}
\hline
No. & \multicolumn{1}{|c|}{$(k,l)$} & \multicolumn{1}{|c|}{$(p,q),(r,s)$} & \multicolumn{1}{|c|}{($L', \mathfrak u$)} & $\rk$ & \multicolumn{1}{|c|}{$\Sigma_G(G/H)$} \\

\hline

\no & $(1,3)$ & $(1,0),(0,1)$ & $(\SL_{2} {\times} \SL_{2},\KK^{2} {\otimes} \KK^2 {\oplus} \KK^2)$ & $4$ & $\alpha_1, \alpha_2, \alpha_{3}, \alpha_4$  \\

\hline

\no & $(1,3)$ & $(0,1),(1,1)$ & $(\SL_{2} {\times} \SL_{2},\KK^{2} {\otimes} \KK^2 {\oplus} \KK^2)$ & $4$ & $\alpha_1, \alpha_2, \alpha_{3}, \alpha_4$  \\

\hline

\no & $(1,4)$ & $(0,1),(1,1)$ & $(\Sp_{4},\KK^{4} {\oplus} \KK^4)$ & $4$ & $\alpha_1 + \alpha_2, \alpha_2 + \alpha_3, \alpha_{3}, \alpha_4$  \\

\hline

\no & $(2,3)$ & $(1,0),(1,1)$ & $(\SL_{2} {\times} \SL_{2},\KK^{2} {\otimes} \KK^2 {\oplus} \KK^2)$ & $4$ & $\alpha_1, \alpha_2, \alpha_{3}, \alpha_4$  \\

\hline

\no & $(2,3)$ & $(0,1),(1,1)$ & $(\SL_{2} {\times} \SL_{2},\KK^{2} {\otimes} \KK^2 {\oplus} \KK^2)$ & $4$ & $\alpha_1, \alpha_2, \alpha_{3}, \alpha_4$  \\

\hline

\no & $(2,4)$ & $(0,1),(1,1)$ & $(\SL_{2} {\times} \SL_{2},\KK^{2} {\otimes} \KK^2 {\oplus} \KK^2)$ & $4$ & $\alpha_1, \alpha_2, \alpha_{3}, \alpha_4$  \\

\hline

\no & $(3,4)$ & $(1,0),(1,1)$ & $(\SL_{3},\KK^{3} {\oplus} \KK^3)$ & $3$ & $\alpha_1, \alpha_2 +\alpha_3, \alpha_4$  \\

\hline
\end{tabular}
\end{center}
\end{table}

\begin{table}[h]
\caption{Cases with $|\Pi \setminus \Pi_L| = 2$ and $|\Psi|=2$ for type~$\mathsf E_6$} \label{table_E6}

\begin{center}
\renewcommand{\tabcolsep}{3pt}%
\begin{tabular}{|c|l|l|l|c|l|}
\hline
No. & \multicolumn{1}{|c|}{$(k,l)$} & \multicolumn{1}{|c|}{$(p,q),(r,s)$} & \multicolumn{1}{|c|}{($L', \mathfrak u$)} & $\rk$ & \multicolumn{1}{|c|}{$\Sigma_G(G/H)$} \\

\hline

\no & $(1,2)$ & $(1,0),(0,1)$ & $(\SL_{5},\KK^{5} {\otimes} \wedge^2 \KK^5)$ & $5$ & $\alpha_1, \alpha_2, \alpha_3 {+} \alpha_4, \alpha_4 {+} \alpha_5, \alpha_5 {+} \alpha_6$ \\

\hline

\no & $(1,2)$ & $(1,0),(1,1)$ & $(\SL_{5},(\KK^{5})^* {\otimes} \wedge^2 \KK^5)$ & $5$ & $\alpha_1 {+} \alpha_3, \alpha_2, \alpha_3 {+} \alpha_4, \alpha_4 {+} \alpha_5, \alpha_6$ \\

\hline

\no & $(1,3)$ & $(0,1),(1,2)$ & $(\SL_{5},(\KK^{5})^* {\otimes} \wedge^2 \KK^5)$ & $5$ & $\alpha_1, \alpha_2, \alpha_3 {+} \alpha_4, \alpha_4 {+} \alpha_5, \alpha_5 {+} \alpha_6$ \\

\hline

\no & $(1,5)$ & $(1,0),(1,1)$ & $(\SL_{4}{\otimes}\SL_2, \KK^{4} {\otimes} \KK^2 {\oplus} (\KK^4)^*$ & $5$ & $\alpha_1, \alpha_3, \alpha_2 {+} \alpha_4, \alpha_4 {+} \alpha_5, \alpha_6$ \\

\hline

\no & $(1,6)$ & $(1,0),(0,1)$ & $(\Spin_{8}, \KK^{8}_+ {\oplus} \KK^8_-)$ & $5$ &
\renewcommand{\tabcolsep}{0pt}%
\begin{tabular}{l}
$\alpha_1, \alpha_2 {+} \alpha_3 {+} \alpha_4, \alpha_2 {+} \alpha_4 {+} \alpha_5$, \\
$\alpha_3 {+} \alpha_4 {+} \alpha_5, \alpha_6$
\end{tabular} \\

\hline

\no & $(1,6)$ & $(1,0),(1,1)$ & $(\Spin_{8}, \KK^{8}_+ {\oplus} \KK^8_-)$ & $5$ &
\renewcommand{\tabcolsep}{0pt}%
\begin{tabular}{l}
$\alpha_1, \alpha_2 {+} \alpha_3 {+} \alpha_4, \alpha_2 {+} \alpha_4 {+} \alpha_5$, \\
$\alpha_3 {+} \alpha_4 {+} \alpha_5, \alpha_6$
\end{tabular} \\

\hline

\no & $(2,3)$ & $(1,0),(0,1)$ & $(\SL_{4}{\otimes}\SL_2, \KK^{4} {\otimes} \KK^2 {\oplus}\KK^4)$ & $5$ & $\alpha_1, \alpha_2 {+} \alpha_4, \alpha_3 {+} \alpha_4, \alpha_5, \alpha_6$ \\

\hline

\no & $(2,3)$ & $(0,1),(1,2)$ & $(\SL_{4}{\otimes}\SL_2, \KK^{4} {\otimes} \KK^2 {\oplus} (\KK^4)^*)$ & $5$ & $\alpha_1, \alpha_2 {+} \alpha_4, \alpha_3, \alpha_4 {+} \alpha_5, \alpha_6$ \\

\hline
\end{tabular}
\end{center}
\end{table}

\begin{table}[h]
\caption{Cases with $|\Pi \setminus \Pi_L| = 2$ and $|\Psi|=2$ for type~$\mathsf E_7$} \label{table_E7}

\begin{center}
\renewcommand{\tabcolsep}{3pt}%
\begin{tabular}{|c|l|l|l|c|l|}
\hline
No. & \multicolumn{1}{|c|}{$(k,l)$} & \multicolumn{1}{|c|}{$(p,q),(r,s)$} & \multicolumn{1}{|c|}{($L', \mathfrak u$)} & $\rk$ & \multicolumn{1}{|c|}{$\Sigma_G(G/H)$} \\

\hline

\no & $(1,2)$ & $(1,0),(0,1)$ & $(\SL_{6}, \KK^{6} {\oplus} \wedge^2 \KK^6)$ & $6$ &
\renewcommand{\tabcolsep}{0pt}%
\begin{tabular}{l}
$\alpha_1 {+} \alpha_3, \alpha_2, \alpha_3 {+} \alpha_4, \alpha_4 {+} \alpha_5$, \\ $\alpha_5 {+} \alpha_6, \alpha_6 {+} \alpha_7$
\end{tabular} \\

\hline

\no & $(1,2)$ & $(0,1),(1,2)$ & $(\SL_{6}, (\KK^{6})^* {\oplus} \wedge^2 \KK^6)$ & $6$ &
\renewcommand{\tabcolsep}{0pt}%
\begin{tabular}{l}
$\alpha_1 {+} \alpha_3, \alpha_2, \alpha_3 {+} \alpha_4, \alpha_4 {+} \alpha_5$, \\ $\alpha_5 {+} \alpha_6, \alpha_6 {+} \alpha_7$
\end{tabular} \\

\hline

\no & $(1,5)$ & $(1,0),(1,1)$ & $(\SL_{4}{\otimes}{\SL_3}, \KK^{4} {\otimes} \KK^3 {\oplus} (\KK^4)^*)$ & $7$ & $\alpha_1, \alpha_2, \alpha_3, \alpha_4, \alpha_5, \alpha_6, \alpha_7$ \\

\hline

\no & $(2,3)$ & $(1,0),(0,1)$ & $(\SL_{5}{\otimes}{\SL_2}, \KK^{5} {\otimes} \KK^2 {\oplus} \KK^5)$ & $5$ & $\alpha_1, \alpha_2 {+} \alpha_4 {+} \alpha_5, \alpha_3 {+} \alpha_4 {+} \alpha_5, \alpha_6, \alpha_7$ \\

\hline

\no & $(2,4)$ & $(0,1),(1,3)$ & $(\SL_{4}{\otimes}{\SL_3}, \KK^{4} {\otimes} \KK^3 {\oplus} (\KK^4)^*)$ & $7$ & $\alpha_1, \alpha_2, \alpha_3, \alpha_4, \alpha_5, \alpha_6, \alpha_7$ \\

\hline

\no & $(2,5)$ & $(1,0),(0,1)$ & $(\SL_{4}{\otimes}{\SL_3}, \KK^{4} {\otimes} \KK^3 {\oplus} \KK^4)$ & $7$ & $\alpha_1, \alpha_2, \alpha_3, \alpha_4, \alpha_5, \alpha_6, \alpha_7$ \\

\hline

\no & $(2,6)$ & $(0,1),(1,2)$ & $(\SL_{5}{\otimes}{\SL_2}, \KK^{5} {\otimes} \KK^2 {\oplus} (\KK^5)^*)$ & $5$ & $\alpha_1, \alpha_2 {+} \alpha_4 {+} \alpha_5, \alpha_3 {+} \alpha_4 {+} \alpha_5, \alpha_6, \alpha_7$ \\

\hline

\no & $(2,7)$ & $(0,1),(1,1)$ & $(\SL_{6}, (\KK^{6})^* {\oplus} \wedge^2 \KK^6)$ & $6$ &
\renewcommand{\tabcolsep}{0pt}%
\begin{tabular}{l}
$\alpha_1 {+} \alpha_3, \alpha_2, \alpha_3 {+} \alpha_4, \alpha_4 {+} \alpha_5$, \\
$\alpha_5 {+} \alpha_6, \alpha_6 {+} \alpha_7$
\end{tabular} \\

\hline

\no & $(3,7)$ & $(0,1),(1,1)$ & $(\SL_{5}{\otimes}{\SL_2}, \KK^{5} {\otimes} \KK^2 {\oplus} (\KK^5)^*)$ & $5$ & $\alpha_1, \alpha_2 {+} \alpha_4 {+} \alpha_5, \alpha_3 {+} \alpha_4 {+} \alpha_5, \alpha_6, \alpha_7$ \\

\hline
\end{tabular}
\end{center}
\end{table}

\begin{table}[h]
\caption{Cases with $|\Pi \setminus \Pi_L| = 2$ and $|\Psi|=2$ for type~$\mathsf E_8$} \label{table_E8}

\begin{center}
\renewcommand{\tabcolsep}{3pt}%
\begin{tabular}{|c|l|l|l|c|l|}
\hline
No. & \multicolumn{1}{|c|}{$(k,l)$} & \multicolumn{1}{|c|}{$(p,q),(r,s)$} & \multicolumn{1}{|c|}{($L', \mathfrak u$)} & $\rk$ & \multicolumn{1}{|c|}{$\Sigma_G(G/H)$} \\

\hline

\no & $(1,2)$ & $(1,0),(0,1)$ & $(\SL_{7}, \KK^{7} {\oplus} \wedge^2 \KK^7)$ & $7$ &
\renewcommand{\tabcolsep}{0pt}%
\begin{tabular}{l}
$\alpha_1, \alpha_2, \alpha_3 {+} \alpha_4, \alpha_4 {+} \alpha_5$, \\ $\alpha_5 {+} \alpha_6, \alpha_6 {+} \alpha_7, \alpha_7 {+} \alpha_8$
\end{tabular} \\

\hline

\no & $(1,3)$ & $(0,1),(1,3)$ & $(\SL_{7}, (\KK^{7})^* {\oplus} \wedge^2 \KK^7)$ & $7$ &
\renewcommand{\tabcolsep}{0pt}%
\begin{tabular}{l}
$\alpha_1, \alpha_2, \alpha_3 {+} \alpha_4, \alpha_4 {+} \alpha_5$, \\ $\alpha_5 {+} \alpha_6, \alpha_6 {+} \alpha_7, \alpha_7 {+} \alpha_8$
\end{tabular} \\

\hline

\no & $(1,5)$ & $(1,0),(1,1)$ & $(\SL_{4}{\otimes}{\SL_4}, \KK^{4} {\otimes} \KK^4 {\oplus} (\KK^4)^*)$ & $8$ & $\alpha_1, \alpha_2, \alpha_3, \alpha_4, \alpha_5, \alpha_6, \alpha_7, \alpha_8$ \\

\hline

\no & $(2,3)$ & $(1,0),(0,1)$ & $(\SL_{6}{\otimes}{\SL_2}, \KK^{6} {\otimes} \KK^2 {\oplus} \KK^6)$ & $5$ &
\renewcommand{\tabcolsep}{0pt}%
\begin{tabular}{l}
$\alpha_1, \alpha_2 {+} \alpha_4 {+} \alpha_5 {+} \alpha_6$, \\
$\alpha_3 {+} \alpha_4 {+} \alpha_5 {+} \alpha_6, \alpha_7, \alpha_8$
\end{tabular} \\

\hline

\no & $(2,5)$ & $(1,0),(0,1)$ & $(\SL_{4}{\otimes}{\SL_4}, \KK^{4} {\otimes} \KK^4 {\oplus} \KK^4)$ & $8$ & $\alpha_1, \alpha_2, \alpha_3, \alpha_4, \alpha_5, \alpha_6, \alpha_7, \alpha_8$ \\

\hline

\no & $(2,5)$ & $(0,1),(1,3)$ & $(\SL_{4}{\otimes}{\SL_4}, \KK^{4} {\otimes} \KK^4 {\oplus} (\KK^4)^*)$ & $8$ & $\alpha_1, \alpha_2, \alpha_3, \alpha_4, \alpha_5, \alpha_6, \alpha_7, \alpha_8$ \\

\hline

\no & $(2,7)$ & $(0,1),(1,2)$ & $(\SL_{6}{\otimes}{\SL_2}, \KK^{6} {\otimes} \KK^2 {\oplus} (\KK^6)^*)$ & $5$ &
\renewcommand{\tabcolsep}{0pt}%
\begin{tabular}{l}
$\alpha_1, \alpha_2 {+} \alpha_4 {+} \alpha_5 {+} \alpha_6$, \\
$\alpha_3 {+} \alpha_4 {+} \alpha_5 {+} \alpha_6, \alpha_7, \alpha_8$
\end{tabular} \\

\hline

\no & $(2,8)$ & $(0,1),(1,1)$ & $(\SL_{7}, (\KK^{7})^* {\oplus} \wedge^2 \KK^7)$ & $7$ &
\renewcommand{\tabcolsep}{0pt}%
\begin{tabular}{l}
$\alpha_1, \alpha_2, \alpha_3 {+} \alpha_4, \alpha_4 {+} \alpha_5$, \\ $\alpha_5 {+} \alpha_6, \alpha_6 {+} \alpha_7, \alpha_7 {+} \alpha_8$
\end{tabular} \\

\hline

\no & $(3,8)$ & $(0,1),(1,1)$ & $(\SL_{6}{\otimes}{\SL_2}, \KK^{6} {\otimes} \KK^2 {\oplus} (\KK^6)^*)$ & $5$ &
\renewcommand{\tabcolsep}{0pt}%
\begin{tabular}{l}
$\alpha_1, \alpha_2 {+} \alpha_4 {+} \alpha_5 {+} \alpha_6$, \\
$\alpha_3 {+} \alpha_4 {+} \alpha_5 {+} \alpha_6, \alpha_7, \alpha_8$
\end{tabular} \\

\hline
\end{tabular}
\end{center}
\end{table}

\subsection{Proofs of Theorems~\ref{thm_par1psi2}(\ref{thm_par1psi2_a}) and~\ref{thm_par2psi2}(\ref{thm_par2psi2_a}) for the classical types and~\texorpdfstring{$\mathsf G_2$}{G\_2}}

In this section, we prove Theorems~\ref{thm_par1psi2}(\ref{thm_par1psi2_a}) and~\ref{thm_par2psi2}(\ref{thm_par2psi2_a}) for the case where $\mathsf X_n$ is one of $\mathsf A_n$ ($n \ge 1$), $\mathsf B_n$ ($n \ge 3$), $\mathsf C_n$ ($n \ge 2$), $\mathsf D_n$ ($n \ge 4$), or~$\mathsf G_2$.

\begin{proof}[Proof of Theorem~\textup{\ref{thm_par1psi2}(\ref{thm_par1psi2_a})}]
It follows from Proposition~\ref{prop_|Psi|=2}(\ref{prop_|Psi|=2_c}) that $p = 1$.
Next, Lemma~\ref{lemma_sum} yields $q \in \lbrace 2, 3 \rbrace$.
Note that in any case $q$ does not exceed the coefficient of~$\alpha_k$ in the expression for the highest root of~$\Delta^+$ as a linear combination of the simple roots.
The fact that the $L$-module $\mathfrak g(\overline \alpha_k) \oplus \mathfrak g (q \overline \alpha_k)$ is spherical implies that so is $\mathfrak g(\overline \alpha_k)$, hence the pair $(\mathsf X_n,k)$ appears in Table~\ref{table_par1psi1}.
If $(\mathsf X_n, k) = (\mathsf G_2, 1)$ and $q=2$, then $\dim \mathfrak g(2\overline \alpha_1) = 1$, hence the $L$-module $\mathfrak g(\overline \alpha_1) \oplus \mathfrak g(2\overline \alpha_1)$ is not strictly indecomposable.
If $(\mathsf X_n, k) = (\mathsf G_2, 1)$ and $q=3$, then $\mathfrak g(\overline \alpha_1) \oplus \mathfrak g(3\overline \alpha_1) \simeq \KK^2 \oplus \KK^2$ as $L'$-modules; since $\dim C = 1$, we see that $\mathfrak g(\overline \alpha_1) \oplus \mathfrak g(3\overline \alpha_1)$ is not spherical as an~$L$-module.
If $\mathsf X_n = \mathsf A_n$, then the highest root of~$\Delta^+$ is $\alpha_1 + \ldots + \alpha_n$, hence $\Phi^+$ cannot contain $C$-roots of the form $q \overline \alpha_k$ for $q \ge 2$.
In the remaining cases the highest root of $\Delta^+$ has only coefficients~$1$ and~$2$, which necessarily implies~$q = 2$.
If $(\mathsf X_n, k) = (\mathsf B_n, 1)$, then $\Phi^+ = \lbrace \overline \alpha_n \rbrace$ and hence $2\overline \alpha_k \notin \Phi^+$.
If $(\mathsf X_n, k) = (\mathsf B_n, n)$, then $\Phi^+ = \lbrace \overline \alpha_n, 2  \overline \alpha_n \rbrace$ and the pair $(L, \mathfrak g(\overline \alpha_n) \oplus \mathfrak g(2 \overline \alpha_n))$ is equivalent to $(\GL_{n}, \KK^n \oplus \wedge^2 \KK^n)$.
Since the latter module is strictly indecomposable and spherical, we get case~\ref{Bnn2} of Table~\ref{table_par1psi2}.
If $(\mathsf X_n,k) = (\mathsf C_n, k)$ with $1 \le k \le n-1$, then $\Phi^+ = \lbrace \overline \alpha_k, 2 \overline \alpha_k \rbrace$ and the pair $(L, \mathfrak g(\overline \alpha_n) \oplus \mathfrak g(2 \overline \alpha_n))$ is equivalent to $(\GL_k \times \Sp_{2n-2k}, \KK^k \otimes \KK^{2n-2k} \oplus \mathrm S^2 \KK^k)$.
The latter module is strictly indecomposable if and only if $k \ge 2$, in which case it is not spherical.
If $(\mathsf X_n, k) = (\mathsf C_n, n)$, then $\Phi^+ = \lbrace \overline \alpha_n \rbrace$ and hence $2\overline \alpha_n \notin \Phi^+$.
If $(\mathsf X_n, k) = (\mathsf D_n, k)$ with $k \in \lbrace 1, n \rbrace$, then $\Phi^+ = \lbrace \overline \alpha_k \rbrace$ and hence $2\overline \alpha_k \notin \Phi^+$.
\end{proof}

\begin{proof}[Proof of Theorem~\textup{\ref{thm_par2psi2}(\ref{thm_par2psi2_a})}]
Let $\lambda, \mu \in \Phi^+$ be two distinct elements such that $\Psi = \lbrace \lambda, \mu \rbrace$.
If $n = 2$, then $\dim \mathfrak g(\nu) = 1$ for all $\nu \in \Phi^+$, hence the $L$-module $\mathfrak g(\lambda) \oplus \mathfrak g(\mu)$ cannot be strictly indecomposable.
So $G$ cannot be of type~$\mathsf G_2$ and in what follows we assume $n \ge 3$.
Since $\Supp \Psi = \Pi$, the elements $\lambda$ and $\mu$ are not proportional to each other, hence $\mathfrak g(\lambda) \oplus \mathfrak g(\mu)$ is a saturated $L$-module.
Thanks to Proposition~\ref{prop_|Psi|=2}(\ref{prop_|Psi|=2_c}), up to interchanging $\lambda$ and~$\mu$ we may assume that $\lambda \in \lbrace \overline \alpha_k, \overline \alpha_l \rbrace$.
Using Lemma~\ref{lemma_sum} we find that, up to interchanging the summands, $\mu = \lambda + (\mu-\lambda)$ is the unique expression of $\mu$ as a sum of two elements of~$\Phi^+$.
In what follows we treat each possibility for~$\mathsf X_n$ separately.

Suppose $\mathsf X_n = \mathsf A_n$ with $n \ge 3$.
Then $\Phi^+ = \lbrace \overline \alpha_k, \overline \alpha_l, \overline \alpha_k + \overline \alpha_l \rbrace$.
We have $L' \simeq \SL_k \times \SL_{l-k} \times \SL_{n+1-l}$; for $i=1,2,3$ let $V_i$ denote the tautological representation of the $i$th factor of~$L'$.
Then, as $L'$-modules, $\mathfrak g(\overline \alpha_k)$, $\mathfrak g(\overline \alpha_l)$, $\mathfrak g(\overline \alpha_k + \overline \alpha_l) \rbrace$ are isomorphic to $V_1 \otimes V_2^*$, $V_2 \otimes V_3^*$, $V_1 \otimes V_3^*$, respectively.
Up to an automorphism of the Dynkin diagram, we may assume that $\lambda = \overline \alpha_k$.
Then $\mu \in \lbrace \overline \alpha_l, \overline \alpha_k + \overline \alpha_l \rbrace$.
If $\mu = \overline \alpha_l$, then the $L$-module $\mathfrak g(\lambda) \oplus \mathfrak g(\mu)$ is strictly indecomposable if and only if $l-k \ge 2$, in which case it is spherical if and only if $l-k = 2$ or $\min(k,n+1-l) = 1$.
Up to an automorphism of the Dynkin diagram, we obtain cases~\ref{An1}, \ref{An2}, \ref{An7} of Table~\ref{table_A}.
If $\mu = \overline \alpha_k + \overline \alpha_l$, then the $L$-module $\mathfrak g(\lambda) \oplus \mathfrak g(\mu)$ is strictly indecomposable if and only if $k \ge 2$, in which case it is spherical if and only if $k = 2$ or $\min(l-k,n+1-l) = 1$.
This yields all the remaining cases in Table~\ref{table_A}.

Suppose $\mathsf X_n = \mathsf B_n$ with $n \ge 3$.
Then
\[
\lbrace \overline \alpha_k, \overline \alpha_l, \overline \alpha_k + \overline \alpha_l, \overline \alpha_k + 2\overline \alpha_l \rbrace \subset \Phi^+ \subset \lbrace \overline \alpha_k, \overline \alpha_l, \overline \alpha_k + \overline \alpha_l, \overline \alpha_k + 2\overline \alpha_l, 2\overline \alpha_l, 2\overline \alpha_k + 2\overline \alpha_l \rbrace,
\]
$2\overline \alpha_l \in \Phi^+$ if and only if $l - k \ge 2$, and $2\overline \alpha_k + 2\overline \alpha_l \in \Phi^+$ if and only if $k \ge 2$.
We have $L \simeq \GL_k \times \GL_{l-k} \times \SO_{2n-2l+1}$; for $i=1,2,3$ let $V_i$ denote the tautological representation of the $i$th factor of~$L$.
Then, as $L$-modules, $\mathfrak g(\overline \alpha_k)$, $\mathfrak g(\overline \alpha_l)$, $\mathfrak g(\overline \alpha_k + \overline \alpha_l)$,  $\mathfrak g(\overline \alpha_k + 2\overline \alpha_l)$, $\mathfrak g(2\overline \alpha_l)$, $\mathfrak g(2\overline \alpha_k + 2\overline \alpha_l)$ are isomorphic to $V_1 \otimes V_2^*$, $V_2 \otimes V_3$, $V_1 \otimes V_3$, $V_1 \otimes V_2$, $\wedge^2 V_2$, $\wedge^2 V_1$, respectively.
In the following, we consider two cases.

Case~1: $\lambda = \overline \alpha_k$.
Then $\mu \in \lbrace \overline \alpha_l, \overline \alpha_k + \overline \alpha_l \rbrace$.
If $\mu = \overline \alpha_l$, then the $L$-module $\mathfrak g(\lambda) \oplus \mathfrak g(\mu)$ is strictly indecomposable if and only if $l-k \ge 2$, in which case it is spherical if and only if $l=n$.
Thus we obtain cases~\ref{Bn1}, \ref{Bn2} of Table~\ref{table_B}.
If $\mu = \overline \alpha_k + \overline \alpha_l$, then the $L$-module $\mathfrak g(\lambda) \oplus \mathfrak g(\mu)$ is strictly indecomposable if and only if $k \ge 2$, in which case it is spherical if and only if $l=n$.
This yields cases~\ref{Bn3}, \ref{Bn4} of Table~\ref{table_B}.

Case~2: $\lambda = \overline \alpha_l$.
We may assume $\mu \ne \overline \alpha_k$, which leaves the possibilities $\mu = \overline \alpha_k + \overline \alpha_l$ or $\mu = \overline \alpha_k + 2\overline \alpha_l$ (the latter one is realized if and only if $2\overline \alpha_l \notin \Phi^+$, which is equivalent to $l-k = 1$).
If $\mu = \overline \alpha_k + \overline \alpha_l$, then the $L$-module $\mathfrak g(\lambda) \oplus \mathfrak g(\mu)$ is strictly indecomposable if and only if $l \le n-1$, in which case it is not spherical.
If $\mu = \overline \alpha_k + 2\overline \alpha_l$ and $l-k=1$, then the $L$-module $\mathfrak g(\lambda) \oplus \mathfrak g(\mu)$ is not strictly indecomposable.

Suppose $\mathsf X_n = \mathsf C_n$ with $n \ge 3$ and $l \le n-1$.
Then
\[
\Phi^+ = \lbrace \overline \alpha_k, \overline \alpha_l, \overline \alpha_k + \overline \alpha_l, \overline \alpha_k + 2\overline \alpha_l, 2\overline \alpha_l, 2\overline \alpha_k + 2\overline \alpha_l \rbrace.
\]
We have $L \simeq \GL_k \times \GL_{l-k} \times \Sp_{2n-2l}$; for $i=1,2,3$ let $V_i$ denote the tautological representation of the $i$th factor of~$L$.
Then, as $L$-modules, $\mathfrak g(\overline \alpha_k)$, $\mathfrak g(\overline \alpha_l)$, $\mathfrak g(\overline \alpha_k + \overline \alpha_l)$,  $\mathfrak g(\overline \alpha_k + 2\overline \alpha_l)$, $\mathfrak g(2\overline \alpha_l)$, $\mathfrak g(2\overline \alpha_k + 2\overline \alpha_l)$ are isomorphic to $V_1 \otimes V_2^*$, $V_2 \otimes V_3$, $V_1 \otimes V_3$, $V_1 \otimes V_2$, $\mathrm S^2 V_2$, $\mathrm S^2 V_1$, respectively.
In the following, we consider two cases.

Case~1: $\lambda = \overline \alpha_k$.
Then $\mu \in \lbrace \overline \alpha_l, \overline \alpha_k + \overline \alpha_l \rbrace$.
If $\mu = \overline \alpha_l$, then the $L$-module $\mathfrak g(\lambda) \oplus \mathfrak g(\mu)$ is strictly indecomposable if and only if $l-k \ge 2$, in which case it is spherical if and only if $l-k=2$ or $k = n-l = 1$.
We obtain cases~\ref{Cn3}--\ref{Cn5}, \ref{Cn12}, \ref{Cn14} in Table~\ref{table_C}.
If $\mu = \overline \alpha_k + \overline \alpha_l$, then the $L$-module $\mathfrak g(\lambda) \oplus \mathfrak g(\mu)$ is strictly indecomposable if and only if $k \ge 2$, in which case it is spherical if and only if $k=2$ or $l-k = n-l = 1$.
This yields cases~\ref{Cn8}--\ref{Cn11}, \ref{Cn15} in Table~\ref{table_C}.

Case~2: $\lambda = \overline \alpha_l$.
We may assume $\mu \ne \overline \alpha_k$, which leaves the only possibility $\mu = \overline \alpha_k + \overline \alpha_l$.
As $l \le n-1$, the $L$-module $\mathfrak g(\lambda) \oplus \mathfrak g(\mu)$ is strictly indecomposable; it is spherical if and only if $l=n-1$ or $k=l-k=1$.
Thus we obtain cases~\ref{Cn1}, \ref{Cn2}, \ref{Cn6}, \ref{Cn13}, \ref{Cn16} in Table~\ref{table_C}.

Suppose $\mathsf X_n = \mathsf C_n$ with $n \ge 3$ and $l = n$.
Then $\Phi^+ = \lbrace \overline \alpha_k, \overline \alpha_n, \overline \alpha_k + \overline \alpha_n, 2\overline \alpha_k + \overline \alpha_n \rbrace$.
We have $L \simeq \GL_k \times \GL_{n-k}$; for $i=1,2$ let $V_i$ denote the tautological representation of the $i$th factor of~$L$.
Then, as $L$-modules, $\mathfrak g(\overline \alpha_k)$, $\mathfrak g(\overline \alpha_n)$, $\mathfrak g(\overline \alpha_k + \overline \alpha_n)$, $\mathfrak g(2\overline \alpha_k + \overline \alpha_n)$ are isomorphic to $V_1 \otimes V_2^*$, $\mathrm{S}^2V_2$, $V_1 \otimes V_2$, $\mathrm S^2 V_1$, respectively.
In the following, we consider two cases.

Case~1: $\lambda = \overline \alpha_k$.
Then $\mu \in \lbrace \overline \alpha_n, \overline \alpha_k + \overline \alpha_n, 2\overline \alpha_k + \overline \alpha_n \rbrace$.
If $\mu = \overline \alpha_n$, then the $L$-module $\mathfrak g(\lambda) \oplus \mathfrak g(\mu)$ is strictly indecomposable if and only if $n-k \ge 2$, in which case it is not spherical.
If $\mu = \overline \alpha_k + \overline \alpha_n$, then the $L$-module $\mathfrak g(\lambda) \oplus \mathfrak g(\mu)$ is strictly indecomposable; it is spherical if and only if $\min(k,n-k) = 1$.
This yields cases~\ref{Cn7}, \ref{Cn17} in Table~\ref{table_C}.
If $\mu = 2\overline \alpha_k + \overline \alpha_n$, then the $L$-module $\mathfrak g(\lambda) \oplus \mathfrak g(\mu)$ is strictly indecomposable if and only if $k \ge 2$, in which case it is not spherical.

Case~2: $\lambda = \overline \alpha_n$.
We may assume $\mu \ne \overline \alpha_k$, which leaves the only possibility $\mu = \overline \alpha_k + \overline \alpha_n$.
Then the $L$-module $\mathfrak g(\lambda) \oplus \mathfrak g(\mu)$ is strictly indecomposable if and only if $n-k \ge 2$, in which case it is not spherical.

Suppose $\mathsf X_n = \mathsf D_n$ with $n \ge 4$ and $l \le n-2$.
Then
\[
\lbrace \overline \alpha_k, \overline \alpha_l, \overline \alpha_k + \overline \alpha_l, \overline \alpha_k + 2\overline \alpha_l \rbrace \subset \Phi^+ \subset \lbrace \overline \alpha_k, \overline \alpha_l, \overline \alpha_k + \overline \alpha_l, \overline \alpha_k + 2\overline \alpha_l, 2\overline \alpha_l, 2\overline \alpha_k + 2\overline \alpha_l \rbrace,
\]
$2\overline \alpha_l \in \Phi^+$ if and only if $l - k \ge 2$, and $2\overline \alpha_k + 2\overline \alpha_l \in \Phi^+$ if and only if $k \ge 2$.
We have $L \simeq \GL_k \times \GL_{l-k} \times \SO_{2n-2l}$; for $i=1,2,3$ let $V_i$ denote the tautological representation of the $i$th factor of~$L$.
Then, as $L$-modules, $\mathfrak g(\overline \alpha_k)$, $\mathfrak g(\overline \alpha_l)$, $\mathfrak g(\overline \alpha_k + \overline \alpha_l)$,  $\mathfrak g(\overline \alpha_k + 2\overline \alpha_l)$, $\mathfrak g(2\overline \alpha_l)$, $\mathfrak g(2\overline \alpha_k + 2\overline \alpha_l)$ are isomorphic to $V_1 \otimes V_2^*$, $V_2 \otimes V_3$, $V_1 \otimes V_3$, $V_1 \otimes V_2$, $\wedge^2 V_2$, $\wedge^2 V_1$, respectively.
In the following, we consider two cases.

Case~1: $\lambda = \overline \alpha_k$.
Then $\mu \in \lbrace \overline \alpha_l, \overline \alpha_k + \overline \alpha_l \rbrace$.
If $\mu = \overline \alpha_l$, then the $L$-module $\mathfrak g(\lambda) \oplus \mathfrak g(\mu)$ is strictly indecomposable if and only if $l-k \ge 2$, in which case it is not spherical.
If $\mu = \overline \alpha_k + \overline \alpha_l$, then the $L$-module $\mathfrak g(\lambda) \oplus \mathfrak g(\mu)$ is strictly indecomposable if and only if $k \ge 2$, in which case it is not spherical.

Case~2: $\lambda = \overline \alpha_l$.
We may assume $\mu \ne \overline \alpha_k$, which leaves the possibilities $\mu = \overline \alpha_k + \overline \alpha_l$ or $\mu = \overline \alpha_k + 2\overline \alpha_l$ (the latter one is realized if and only if $2\overline \alpha_l \notin \Phi^+$, which is equivalent to $l-k = 1$).
If $\mu = \overline \alpha_k + \overline \alpha_l$, then the $L$-module $\mathfrak g(\lambda) \oplus \mathfrak g(\mu)$ is strictly indecomposable but not spherical.
If $\mu = \overline \alpha_k + 2\overline \alpha_l$ and $l-k=1$, then the $L$-module $\mathfrak g(\lambda) \oplus \mathfrak g(\mu)$ is not strictly indecomposable.

Suppose $\mathsf X_n = \mathsf D_n$ with $n \ge 4$, $k \le n-2$, and $l \in \lbrace n-1,n \rbrace$.
Up to an automorphism of the Dynkin diagram, we may assume that $l = n$.
Then
\[
\lbrace \overline \alpha_k, \overline \alpha_n, \overline \alpha_k + \overline \alpha_n \rbrace \subset \Phi^+ \subset \lbrace \overline \alpha_k, \overline \alpha_n, \overline \alpha_k + \overline \alpha_n, 2\overline \alpha_k + \overline \alpha_n \rbrace
\]
and $2\overline \alpha_k + \overline \alpha_n \in \Phi^+$ if and only if $k \ge 2$.
We have $L \simeq \GL_k \times \GL_{n-k}$; for $i=1,2$ let $V_i$ denote the tautological representation of the $i$th factor of~$L$.
Then, as $L$-modules, $\mathfrak g(\overline \alpha_k)$, $\mathfrak g(\overline \alpha_n)$, $\mathfrak g(\overline \alpha_k + \overline \alpha_n)$, $\mathfrak g(2\overline \alpha_k + \overline \alpha_n)$ are isomorphic to $V_1 \otimes V_2^*$, $\wedge^2 V_2$, $V_1 \otimes V_2$, $\wedge^2 V_1$, respectively.
In the following, we consider two cases.

Case~1: $\lambda = \overline \alpha_k$.
Then $\mu \in \lbrace \overline \alpha_n, \overline \alpha_k + \overline \alpha_n, 2\overline \alpha_k + \overline \alpha_n \rbrace$.
If $\mu = \overline \alpha_n$, then the $L$-module $\mathfrak g(\lambda) \oplus \mathfrak g(\mu)$ is strictly indecomposable if and only if $n-k \ge 3$, in which case it is spherical if $k = 1$ or $n - k = 3$.
Then we obtain cases~\ref{Dn1}, \ref{Dn2}, \ref{Dn5}, \ref{Dn9} in Table~\ref{table_D}.
If $\mu = \overline \alpha_k + \overline \alpha_n$, then the $L$-module $\mathfrak g(\lambda) \oplus \mathfrak g(\mu)$ is strictly indecomposable; as $n-k \ge 2$, it is spherical if and only if $k = 1$.
This yields case~\ref{Dn3} in Table~\ref{table_D}.
If $\mu = 2\overline \alpha_k + \overline \alpha_n$, then the $L$-module $\mathfrak g(\lambda) \oplus \mathfrak g(\mu)$ is strictly indecomposable if and only if $k \ge 3$; as $n-k \ge 2$, it is spherical if and only if $k = 3$.
This yields cases~\ref{Dn7}, \ref{Dn8} in Table~\ref{table_D}.

Case~2: $\lambda = \overline \alpha_n$.
We may assume $\mu \ne \overline \alpha_k$, which leaves the only possibility $\mu = \overline \alpha_k + \overline \alpha_n$.
Then the $L$-module $\mathfrak g(\lambda) \oplus \mathfrak g(\mu)$ is strictly indecomposable if and only if $n-k \ge 3$; in which case it is spherical if and only if $k=1$ or $n-k = 3$.
Thus we obtain cases~\ref{Dn4}, \ref{Dn6}, \ref{Dn10} in Table~\ref{table_D}.

Suppose $\mathsf X_n = \mathsf D_n$ with $n \ge 4$, $k = n-1$, and $l = n$.
Then $\Phi^+ = \lbrace \overline \alpha_{n-1}, \overline \alpha_n, \overline \alpha_{n-1} + \overline \alpha_n \rbrace$.
We have $L \simeq \GL_{n-1} \times \GL_1$; for $i=1,2$ let $V_i$ denote the tautological representation of the $i$th factor of~$L$.
Then, as $L$-modules, $\mathfrak g(\overline \alpha_{n-1})$, $\mathfrak g(\overline \alpha_n)$, $\mathfrak g(\overline \alpha_{n-1} + \overline \alpha_n) \rbrace$ are isomorphic to $V_1 \otimes V_2^*$, $V_1 \otimes V_2$, $\wedge^2 V_1$, respectively.
Up to an automorphism of the Dynkin diagram, we may assume that $\lambda = \overline \alpha_{n-1}$.
Then $\mu \in \lbrace \overline \alpha_n, \overline \alpha_{n-1} + \overline \alpha_n \rbrace$.
If $\mu = \overline \alpha_n$, then the $L$-module $\mathfrak g(\lambda) \oplus \mathfrak g(\mu)$ is strictly indecomposable and spherical.
Then we obtain case~\ref{Dn11} in Table~\ref{table_D}.
If $\mu = \overline \alpha_{n-1} + \overline \alpha_n$, then the $L$-module $\mathfrak g(\lambda) \oplus \mathfrak g(\mu)$ is also strictly indecomposable and spherical.
This yields cases~\ref{Dn12}, \ref{Dn13} in Table~\ref{table_D}.
\end{proof}

\subsection{Proofs of Theorems~\ref{thm_par1psi2}(\ref{thm_par1psi2_b}) and~\ref{thm_par2psi2}(\ref{thm_par2psi2_b}) for the classical types}

In each of the cases classified in Theorems~\ref{thm_par1psi2}(\ref{thm_par1psi2_a}) and~\ref{thm_par2psi2}(\ref{thm_par2psi2_a}), the computation of the set $\Sigma_G(G/H)$ for the corresponding spherical subgroup $H \subset G$ with $\Psi = \lbrace \lambda_1, \lambda_2 \rbrace$ is done according to the following strategy.
First, by Proposition~\ref{prop_num_sr} one has $|\Sigma_G(G/H)| = \rk_G(\mathfrak u)$, and we take the latter number from~\cite[\S\,5]{Kn98}.
Second, we compute the subgroups $N_1 = N(\lambda_1)$ and~$N_2 = N(\lambda_2)$ (see~\S\,\ref{subsec_degen_add}).
Thanks to Theorem~\ref{thm_degen}(\ref{thm_degen_c}) and Proposition~\ref{prop_sr_union}, one has $|\Sigma_G(G/N_1)| = |\Sigma_G(G/N_2)| = |\Sigma_G(G/H)| - 1$ and $\Sigma_G(G/H) = \Sigma_G(G/N_1) \cup \Sigma_G(G/N_2)$, so it remains to compute $\Sigma_G(G/N_1)$ and $\Sigma_G(G/N_2)$ separately.
Third, for $i=1,2$ we identify all components of the SM-decomposition of $\Psi(N_i)$ and apply Algorithm~\ref{alg_H_i} for each of them.
As a result, we obtain a collection of new spherical subgroups such that each of them has smaller rank and trivial SM-decomposition.
Then the computation of $\Sigma_G(G/H)$ is completed by induction.

Below we provide two examples of implementing the above algorithm.
The abbreviation ``case~M.N($n$)'' refers to case~N in Table~M with rank of~$G$ equal to~$n$.
In the abbreviation~``Case~M.N($n,k$)'', the meanings of M, N, and $n$ are the same as above and $k$ is the additional parameter appearing in the corresponding case.
For $i=1,2$, $\Pi_i$ is the set of simple roots of the Levi subgroup of~$N_i$, $\Psi_i = \Psi(N_i)$ is the corresponding set presented in the form of its SM-decomposition, and $N_{ij}$ is the subgroup obtained from $N_i$ by applying Algorithm~\ref{alg_H_i} for the $j$th component of the SM-decomposition of~$\Psi_i$.
For each subgroup~$N_{ij}$, the notations $\Pi_{ij}$ and $\Psi_{ij}$ have similar meanings.

\begin{example}
Case~\ref{table_par1psi2}.\ref{Bnn2}$(n)$, $n \ge 3$. We have $\Pi \setminus \Pi_L = \lbrace \alpha_n \rbrace$ and $\Psi = \lbrace \overline \alpha_n, 2\overline \alpha_n \rbrace$.

Data for~$N_1$: $\Pi \setminus \Pi_1 = \lbrace \alpha_1, \alpha_n \rbrace$; $\Psi_1  = \lbrace \overline \alpha_n, 2\overline \alpha_n \rbrace$ for $n \ge 4$, $\Psi_1 = \lbrace \overline \alpha_n \rbrace \cup \lbrace 2\overline \alpha_n \rbrace$ for $n = 3$.
If $n \ge 4$, then, up to reduction of the ambient group, we get case~\ref{table_par1psi2}.\ref{Bnn2}$(n-1)$.

Data for $N_{11}$ ($n=3$): $\Pi \setminus \Pi_{11} = \lbrace \alpha_1, \alpha_3 \rbrace$, $\Psi_{11} = \lbrace \overline \alpha_3 \rbrace$.
Up to reduction of the ambient group, we get case~\ref{table_par1psi1}.\ref{No4}$(2)$, so $\Sigma_G(G/N_{11}) = \lbrace \alpha_2 + \alpha_3 \rbrace$.

Data for $N_{12}$ ($n=3$): $\Pi \setminus \Pi_{12} = \lbrace \alpha_1, \alpha_2, \alpha_3 \rbrace$, $\Psi_{12} = \lbrace \overline \alpha_3 \rbrace$.
Up to reduction of the ambient group, we get case~\ref{table_par1psi1}.\ref{No1}$(1,1)$, so $\Sigma_G(G/N_{12}) = \lbrace \alpha_3 \rbrace$.

Data for~$N_2$: $\Pi \setminus \Pi_2 = \lbrace \alpha_2, \alpha_n \rbrace$, $\Psi_2 = \lbrace \overline \alpha_2 + \overline \alpha_n \rbrace \cup \lbrace \overline \alpha_n, 2\overline \alpha_n \rbrace$ for $n \ge 5$, $\Psi = \lbrace \overline \alpha_2 + \overline \alpha_n \rbrace \cup \lbrace \overline \alpha_n \rbrace \cup \lbrace 2\overline \alpha_n \rbrace$ for $n = 4$, and $\lbrace \overline \alpha_2 + \overline \alpha_n \rbrace \cup \lbrace \overline \alpha_n \rbrace$ for $n = 3$.

Data for~$N_{21}$ ($n \ge 4$): $\Pi \setminus \Pi_{21} = \lbrace \alpha_2, \alpha_3, \alpha_n \rbrace$, $\Psi_{21} = \lbrace \overline \alpha_2 \rbrace$. Up to reduction of the ambient group, we get case~\ref{table_par1psi1}.\ref{No1}$(2,1)$, so $\Sigma_G(G/N_{21}) = \lbrace \alpha_1 + \alpha_2 \rbrace$.

Data for~$N_{21}$ ($n = 3$): $\Pi \setminus \Pi_{21} = \lbrace \alpha_2, \alpha_3 \rbrace$, $\Psi_{21} = \lbrace \overline \alpha_2 \rbrace$. Up to reduction of the ambient group, we again get case~\ref{table_par1psi1}.\ref{No1}$(2,1)$, so $\Sigma_G(G/N_{21}) = \lbrace \alpha_1 + \alpha_2 \rbrace$.

Data for~$N_{22}$ ($n \ge 5$): $\Pi \setminus \Pi_{22} = \lbrace \alpha_2, \alpha_n \rbrace$, $\Psi_{22} = \lbrace \overline \alpha_n, 2\overline \alpha_n \rbrace$.
Up to reduction of the ambient group, we get the same case~\ref{table_par1psi2}.\ref{Bnn2}$(n-2)$.

Data for~$N_{22}$ ($n = 4$): $\Pi \setminus \Pi_{22} = \lbrace \alpha_2, \alpha_4 \rbrace$, $\Psi_{22} = \lbrace \overline \alpha_4 \rbrace$.
Up to reduction of the ambient group, we get case~\ref{table_par1psi1}.\ref{No4}$(2)$, so $\Sigma_G(G/N_{11}) = \lbrace \alpha_3 + \alpha_4 \rbrace$.

Data for~$N_{22}$ ($n = 3$): $\Pi \setminus \Pi_{22} = \lbrace \alpha_2, \alpha_3 \rbrace$, $\Psi_{22} = \lbrace \overline \alpha_3 \rbrace$.
Up to reduction of the ambient group, we get case~\ref{table_par1psi1}.\ref{No1}$(1,1)$, so $\Sigma_G(G/N_{22}) = \lbrace \alpha_3 \rbrace$.

Data for~$N_{23}$ ($n = 4$): $\Pi \setminus \Pi_{23} = \lbrace \alpha_2, \alpha_3, \alpha_4 \rbrace$, $\Psi_{23} = \lbrace \overline \alpha_4 \rbrace$.
Up to reduction of the ambient group, we get case~\ref{table_par1psi1}.\ref{No1}$(1,1)$, so $\Sigma_G(G/N_{23}) = \lbrace \alpha_4 \rbrace$.

Using the information given above, one proves by induction on~$n$ that $\Sigma_G(G/H) = \lbrace \alpha_i + \alpha_{i+1} \mid 1 \le i \le n-1 \rbrace \cup \lbrace \alpha_n \rbrace$.
\end{example}

\begin{example}
Case~\ref{table_A}.\ref{An1}$(n,k)$, $n-k > k \ge 1$. We have $\Pi \setminus \Pi_L = \lbrace \alpha_k, \alpha_n \rbrace$ and $\Psi = \lbrace \overline \alpha_k, \overline \alpha_n \rbrace$.

Data for~$N_1$ ($k \ge 2$): $\Pi \setminus \Pi_1 = \lbrace \alpha_1, \alpha_k, \alpha_{n-1}, \alpha_n \rbrace$, $\Psi_1 = \lbrace \overline \alpha_k, \overline \alpha_{n-1} + \overline \alpha_n \rbrace \cup \lbrace \overline \alpha_n \rbrace$.

Data for~$N_1$ ($k = 1$): $\Pi \setminus \Pi_1 = \lbrace \alpha_1, \alpha_{n-1}, \alpha_n \rbrace$, $\Psi_1 = \lbrace \overline \alpha_{n-1} + \overline \alpha_n \rbrace \cup \lbrace \overline \alpha_n \rbrace$.

Data for~$N_{11}$ ($k \ge 2$): $\Pi \setminus \Pi_{11} = \lbrace \alpha_1, \alpha_k, \alpha_{n-1}, \alpha_n \rbrace$, $\Psi_{11} = \lbrace \overline \alpha_k, \overline \alpha_{n-1} \rbrace$.
Up to reduction of the ambient group, we get case~\ref{table_A}.\ref{An1}$(n-2,k-1)$.

Data for~$N_{11}$ ($k = 1$): $\Pi \setminus \Pi_{11} = \lbrace \alpha_1, \alpha_{n-1}, \alpha_n \rbrace$, $\Psi_{11} = \lbrace \overline \alpha_{n-1} \rbrace$.
Up to reduction of the ambient group, we get case~\ref{table_par1psi1}.\ref{No1}$(n-2,1)$, so $\Sigma_G(G/N_{11}) = \lbrace \alpha_2 + \ldots + \alpha_{n-1} \rbrace$.

Data for~$N_{12}$ ($k \ge 2$): $\Pi \setminus \Pi_{12} = \lbrace \alpha_1, \alpha_k, \alpha_{n-1}, \alpha_n \rbrace$, $\Psi_{12} = \lbrace \overline \alpha_{n} \rbrace$.
Up to reduction of the ambient group, we get case~\ref{table_par1psi1}.\ref{No1}$(1,1)$, so $\Sigma_G(G/N_{12}) = \lbrace \alpha_n \rbrace$.

Data for~$N_{12}$ ($k = 1$): $\Pi \setminus \Pi_{12} = \lbrace \alpha_1, \alpha_{n-1}, \alpha_n \rbrace$, $\Psi_{12} = \lbrace \overline \alpha_{n} \rbrace$.
Up to reduction of the ambient group, we get case~\ref{table_par1psi1}.\ref{No1}$(1,1)$, so $\Sigma_G(G/N_{12}) = \lbrace \alpha_n \rbrace$.

Data for~$N_2$ ($k \ge 2$): $\Pi \setminus \Pi_2 = \lbrace \alpha_k, \alpha_{k+1}, \alpha_n \rbrace$, $\Psi_2 = \lbrace \overline \alpha_k, \overline \alpha_k + \overline \alpha_{k+1} \rbrace$.
Up to reduction of the ambient group, we get case~\ref{table_A}.\ref{No5}($n-1,k$).

Data for~$N_2$ ($k = 1$): $\Pi \setminus \Pi_2 = \lbrace \alpha_1, \alpha_{2}, \alpha_n \rbrace$, $\Psi_2 = \lbrace \overline \alpha_1 \rbrace \cup \lbrace \overline \alpha_1 + \overline \alpha_{2} \rbrace$.

Data for~$N_{21}$ ($k = 1$): $\Pi \setminus \Pi_{21} = \lbrace \alpha_1, \alpha_{2}, \alpha_n \rbrace$, $\Psi_{21} = \lbrace \overline \alpha_1 \rbrace$.
Up to reduction of the ambient group, we get case~\ref{table_par1psi1}.\ref{No1}($1,1$), so $\Sigma_G(G/N_{21}) = \lbrace \alpha_1 \rbrace$.

Data for~$N_{22}$ ($k = 1$): $\Pi \setminus \Pi_{22} = \lbrace \alpha_1, \alpha_{2}, \alpha_n \rbrace$, $\Psi_{22} = \lbrace \overline \alpha_2 \rbrace$.
Up to reduction of the ambient group, we get case~\ref{table_par1psi1}.\ref{No1}$(n-2,1)$, so $\Sigma_G(G/N_{22}) = \lbrace \alpha_2 + \ldots + \alpha_{n-1} \rbrace$.

We see that the computation of $\Sigma_G(G/H)$ reduces to the same case~\ref{table_A}.\ref{An1} and also to case~\ref{table_A}.\ref{No5} with smaller values of the rank of~$G$, which enables one to complete the computation by induction.
\end{example}

\subsection{Proofs of Theorems~\ref{thm_par1psi2} and~\ref{thm_par2psi2} for the types \texorpdfstring{$\mathsf F_4$}{F\_4}, \texorpdfstring{$\mathsf E_6$}{E\_6}, \texorpdfstring{$\mathsf E_7$}{E\_7}, \texorpdfstring{$\mathsf E_8$}{E\_8}}

For each of the types $\mathsf F_4$, $\mathsf E_6$, $\mathsf E_7$, $\mathsf E_8$, both theorems are proved as results of case-by-case considerations implemented in Python according to the algorithm presented below.
In each case, we realize the sets $\Delta^+,\Pi$ as subsets in $\QQ^n$ as described in~\cite[\S\,12.1]{Hum2} ($n = 4$ for $\mathsf F_4$ and $n = 8$ otherwise). 

\smallskip

Algorithm~\newalg: \label{alg_D}

Step~\step: \label{Dstep1}
consider all possible parabolic subgroups $P \supset B^-$ with standard Levi subgroup~$L$; every such $P$ is determined by the subset $\Pi \setminus \Pi_L$ satisfying $|\Pi \setminus \Pi_L| = 1$ in the case of Theorem~\ref{thm_par1psi2} and $|\Pi \setminus \Pi_L| = 2$ in the case of Theorem~\ref{thm_par2psi2}.

The next steps are applied to a fixed choice of~$P$.

Step~\step: \label{Dstep2}
compute the set $\Phi^+$ and find all subsets $\Psi \subset \Phi^+$ such that $\Psi = \lbrace \lambda, \mu \rbrace$, where $\lambda \ne \mu$, $\lambda = \overline \alpha$ for some $\alpha \in \Pi \setminus \Pi_L$, and $\mu = \lambda + (\mu - \lambda)$ is the unique expression of~$\mu$ as a sum of two elements in~$\Phi^+$.

The next steps are applied to a fixed choice of~$\Psi$. Let $H = L \rightthreetimes H_u \subset P$ be the subgroup with~$\Psi(H) = \Psi$.

Step~\step: \label{Dstep3}
if $|\lbrace \delta \in \Delta^+ \mid \overline \delta = \lambda \rbrace| = 1$ or $|\lbrace \delta \in \Delta^+ \mid \overline \delta = \mu \rbrace| = 1$, then exit; otherwise continue.

Step~\step: \label{Dstep4}
apply Algorithm~\ref{alg_sph_mod} to the triple $(\Pi_L, \Delta^+_L, \lbrace \delta\in \Delta^+ \mid \overline \delta \in \lbrace \lambda, \mu \rbrace\rbrace)$ and let $\Theta$ denote the output.

Step~\step: \label{Dstep5}
if $\Theta$ is linearly dependent, then exit; otherwise continue.

Step~\step: \label{Dstep6}
apply Algorithm~\ref{alg_ba} to~$H$.

\smallskip

According to Proposition~\ref{prop_|Psi|=2}(\ref{prop_|Psi|=2_c}) and Lemma~\ref{lemma_sum}, at step~\ref{Dstep2} we obtain a list of all possible subgroups~$H \subset G$ with $|\Psi(H)|=2$.
Then at step~\ref{Dstep3} we exclude all subgroups~$H$ such that $\dim \mathfrak g(\lambda) = 1$ for some $\lambda \in \Psi(H)$: the SM-decomposition of~$\Psi(H)$ is nontrivial for such cases.
At steps~\ref{Dstep4} and~\ref{Dstep5} we find all cases where $H$ is spherical; see Proposition~\ref{prop_sph_lin_ind}.
Finally, at step~\ref{Dstep6} for every spherical~$H$ we compute several additional spherical subgroups $N_i$ with $\Psi(N_i) = 1$ such that $\Sigma_G(G/H)$ is the union of all sets $\Sigma_G(G/N_i)$.
Up to reduction of the ambient group, the sets $\Sigma_G(G/N_i)$ are then determined by Theorem~\ref{thm_par1psi1}.

After machine execution of Algorithm~\ref{alg_D}, we manually exclude all subgroups~$H$ where the SM-decomposition of~$\Psi(H)$ is nontrivial.
As a result, we get case~\ref{F433} in Table~\ref{table_par1psi2} and all the data in Tables~\ref{table_F4}--\ref{table_E8}.

\renewcommand{\arraystretch}{1}%

\newpage

%%%%%%%%%%%%%%%%%%%%%%%%%%%%%%%%%%%%%%%%%%%%%%%%%%%%%%%%%%%%%%%%%%%%%%%%%%%%%%%%%%%%%%%%%%

\end{document}